\documentclass[11pt]{amsart}
\usepackage[T1]{fontenc}
\usepackage{lmodern}
\usepackage{microtype}
\usepackage{amsmath,amssymb,amsthm,mathtools}
\usepackage{enumitem}
\usepackage{mathrsfs}
\usepackage{hyperref}
\usepackage[capitalize,nameinlink]{cleveref}
\usepackage[a4paper,margin=1.15in]{geometry}
\newcommand{\pr}{\operatorname{pr}}
\DeclareMathOperator{\coker}{coker}
\newcommand{\Mat}{\mathrm{Mat}}

\hypersetup{colorlinks=true,linkcolor=blue,citecolor=blue,urlcolor=blue}
\setlist[enumerate]{itemsep=2pt,topsep=4pt}

\newtheorem{theorem}{Theorem}[section]
\newtheorem{proposition}[theorem]{Proposition}
\newtheorem{lemma}[theorem]{Lemma}
\newtheorem{corollary}[theorem]{Corollary}
\theoremstyle{definition}
\newtheorem{definition}[theorem]{Definition}
\newcommand{\GL}{\operatorname{GL}}
\newtheorem{example}[theorem]{Example}

\theoremstyle{remark}
\newtheorem{remark}[theorem]{Remark}

\newcommand{\PP}{\mathbf P}
\newcommand{\OO}{\mathcal O}
\newcommand{\Ulr}{\operatorname{Ulr}}
\newcommand{\Rep}{\operatorname{Rep}}
\newcommand{\VB}{\operatorname{VB}}
\newcommand{\End}{\operatorname{End}}
\newcommand{\Hom}{\operatorname{Hom}}
\newcommand{\Ext}{\operatorname{Ext}}
\newcommand{\Pic}{\operatorname{Pic}}
\newcommand{\Br}{\operatorname{Br}}
\newcommand{\Spec}{\operatorname{Spec}}
\newcommand{\Sym}{\operatorname{Sym}}
\newcommand{\rk}{\operatorname{rk}}
\newcommand{\Cl}{\operatorname{Cl}}
\newcommand{\ind}{\operatorname{ind}}
\newcommand{\per}{\operatorname{per}}
\newcommand{\charac}{\operatorname{char}}
\newcommand{\uc}{\operatorname{uc}}
\newcommand{\ucnor}{\operatorname{uc}_{\mathrm{nor}}}
\newcommand{\cc}{\operatorname{cc}}
\newcommand{\ccnor}{\operatorname{cc}_{\mathrm{nor}}}
\newcommand{\MFpow}{\operatorname{MF}^{\mathrm{pow}}}
\newcommand{\UlrCl}{\operatorname{Ulr}^{\mathrm{Cl}}}
\newcommand{\id}{\mathrm{id}}

\title[Relative Ulrich bundles and Clifford algebras]{On Relative Ulrich Bundles and Generalized Clifford Algebras}
\author{Soham Mondal}
\author{Anindya Mukherjee}
\date{}


\begin{document}
\begin{abstract}
We give a Clifford-theoretic description of relatively Ulrich bundles
on smooth families of hypersurfaces.  A split projection centre
produces an explicit monic Clifford algebra whose locally free
representations are exactly the relatively Ulrich bundles.  For
quadric fibrations the description is intrinsic and requires no
projection centre; the even Clifford algebra gives exact formulas for
the minimal relative Ulrich rank in terms of Brauer classes and
monodromy.  We also obtain applications to stability, moduli, and
relative wildness.
\end{abstract}
\maketitle
\section{Introduction}

\subsection{Ulrich bundles and Clifford algebras}

Let
\[
Z\subset \mathbf P^N_k
\]
be a projective variety of dimension \(m\), equipped with the
polarization \(\OO_Z(1)\).  A vector bundle \(G\) on \(Z\) is
\emph{Ulrich} if
\[
H^i(Z,G(-j))=0
\qquad
\text{for all }i
\text{ and }1\le j\le m.
\]
Equivalently, for the given projective embedding, the corresponding
graded module has a linear free resolution.  Ulrich modules arose in
commutative algebra and were developed in connection with maximally
generated maximal Cohen--Macaulay modules; see
\cite{Ulr84,BHU87}.  Their geometric study was
greatly developed through linear resolutions, Chow forms and vector
bundles; see, for example, \cite{ESW03,Bea18}.  The general
existence problem for Ulrich bundles is still open.  The minimal
possible rank is therefore an important invariant; compare the notion
of Ulrich complexity in \cite{BES17}.

For hypersurfaces there is a classical connection with generalized
Clifford algebras.  Consider
\[
Z=\{w^d=g(u_1,\ldots,u_n)\}\subset \mathbf P^n_k
\]
and projection
\[
\lambda:Z\longrightarrow \mathbf P^{n-1}_k
\]
obtained by forgetting the coordinate \(w\).  If \(G\) is an Ulrich
bundle of rank \(r\), then the projection criterion gives
\[
\lambda_*G\simeq
\OO_{\mathbf P^{n-1}}^{\oplus dr}.
\]
Multiplication by \(w\) is therefore represented by a
\(dr\times dr\) matrix \(M(u)\) of linear forms satisfying
\[
M(u)^d=g(u)I.
\]
This is a representation of the generalized Clifford algebra of
\(g\).  Generalized Clifford algebras of forms go back to Roby and
Childs \cite{Rob69,Chi78}; their representation theory and its
geometric applications were developed further by Van den Bergh,
Haile, Kulkarni, and Coskun--Kulkarni--Mustopa
\cite{VdB87,Hai84,Kul03,CKM12,CKM13}.

A more general finite-morphism framework was developed by Krashen and
Lieblich \cite{KL}.  Under their hypotheses, the Clifford algebra of a
finite morphism has a representation--Ulrich correspondence, and the
construction is compatible with base change.  Thus the abstract
existence of a Clifford algebra attached to a finite morphism is not
the new point of the present paper.  Our aim is to give an explicit
relative hypersurface model and then use it to study phenomena that
occur only in families.

\subsection{The questions studied in this paper}

We study a smooth family of hypersurfaces
\[
p:Y_f\longrightarrow X
\]
over a positive-dimensional base.  The main questions are the
following.

\begin{enumerate}
\item[(Q1)] Can relatively Ulrich bundles be described by a Clifford
algebra over the base \(X\)?

\item[(Q2)] If every geometric fibre admits an Ulrich bundle of rank
\(r\), must the family admit a relatively Ulrich bundle of the same
rank?

\item[(Q3)] Which features of the relative category cannot be detected
from a single fibre?
\end{enumerate}

Our answers are as follows.  For (Q1), after choosing a split relative
projection centre we obtain an explicit monic Clifford algebra
\(C_{f,\phi}\) and an equivalence
\[
\Ulr(Y_f/X)\simeq \Rep_X(C_{f,\phi}).
\]
The representations must be allowed to live on arbitrary vector
bundles on \(X\); restricting to trivial coefficient bundles sees only
a smaller, normalized subcategory.

For quadrics, the answer is stronger: no projection centre is needed.
The whole relative Ulrich category is described by the intrinsic even
Clifford algebra
\[
\Cl_0(E,f,L)
\]
of the line-bundle-valued quadratic form.  This gives existence of a
relatively Ulrich bundle on every smooth quadric fibration in
characteristic different from \(2\).

The answer to (Q2) is negative in general.  For quadric fibrations we
compute the minimal relative rank exactly.  The difference between
the fibrewise and relative ranks is measured by a Brauer class and,
in the other parity, by monodromy of the discriminant double cover.

For (Q3), we find several genuinely relative phenomena.  The
coefficient bundle can be nontrivial even in a constant family.
Monodromy can exchange the two spinor types on the fibres.  Brauer
classes can force the minimal global rank to increase.  Finally, the
geometry of the base can make the relative Ulrich category strictly
wild even when the variation becomes invisible after restriction to
every geometric fibre.

\subsection{The relative setting}
Throughout the paper, $k$ is an algebraically closed field and $X$
is a locally Noetherian $k$-scheme. Let $E$ be a vector bundle of
rank $n+1$, let $L$ be a line bundle on $X$, and let
\[
\pi:\mathbf P_X(E)\longrightarrow X
\]
be the projective bundle of lines, so that
\[
\pi_*\OO_{\mathbf P_X(E)}(1)=E^\vee.
\]
Let
\[
f\in H^0\!\left(X,\Sym^dE^\vee\otimes L\right)
\]
and let
\[
Y_f:=V(f)\subset\mathbf P_X(E).
\]

We assume
\[
d\geq2,\qquad n\geq2,
\]
that the restriction of $f$ to every geometric projective fibre is
nonzero, and that
\[
p:Y_f\longrightarrow X
\]
is smooth. Hence $p$ is a smooth family of degree-$d$ hypersurfaces
of relative dimension $n-1\geq1$.

Whenever we use the scalar complexity invariants
\[
\uc,\qquad \ucnor,\qquad \cc,\qquad \ccnor,
\]
we assume that the relevant base is nonempty and connected. On a
disconnected base, the same invariants should be considered
componentwise.

A vector bundle \(F\) on \(Y_f\) is \emph{relatively Ulrich} if
\[
Rp_*F(-j)=0,
\qquad
1\le j\le n-1.
\]
By cohomology and base change, this is equivalent to requiring
\(F_{\bar x}\) to be an ordinary Ulrich bundle on \(Y_{\bar x}\) for
every geometric point \(\bar x\to X\).

The first basic difference from the absolute case is that the
coefficient bundle is not usually trivial.  A relatively Ulrich bundle
of rank \(r\) has a relative linear presentation
\[
0\longrightarrow
\pi^*B_1(-1)
\longrightarrow
\pi^*B_0
\longrightarrow
i_*F
\longrightarrow0,
\]
where
\[
B_0\simeq p_*F,
\qquad
\rk B_0=\rk B_1=dr.
\]
There is no reason for \(B_0\) to be trivial.  We call \(F\)
\emph{normalized} when
\[
p_*F\simeq\OO_X^{\oplus dr}.
\]
This distinction is essential in the relative theory.

\subsection{Projection centres and the relative Clifford dictionary}

A relative projection centre is a line subbundle
\[
\phi:S\hookrightarrow E
\]
whose associated section of \(\mathbf P_X(E)\) is disjoint from
\(Y_f\).  We call it split if \(S\) is a direct summand of \(E\).
Centres exist locally on \(X\), while a global centre is a genuine
extra hypothesis.

Fix a split centre.  Put
\[
V=E/S,
\qquad
W=V\otimes S^{-1}.
\]
After choosing a splitting, the equation becomes monic:
\[
\widetilde f
=
w^d+g_1w^{d-1}+\cdots+g_d,
\qquad
g_j\in H^0(X,\Sym^jW^\vee).
\]
Projection from the centre gives a finite flat morphism
\[
q_\phi:Y_f\longrightarrow \mathbf P_X(V)
\]
of degree \(d\).

We define the relative monic Clifford algebra \(C_{f,\phi}\) to be the
\(\OO_X\)-algebra generated by \(W\) subject to the coefficients of
the identity
\[
M(y)^d+g_1(y)M(y)^{d-1}+\cdots+g_d(y)I=0.
\]
Changing the splitting changes the monic coordinate by translation.
The resulting algebras are canonically isomorphic, so
\(C_{f,\phi}\), up to canonical isomorphism, depends only on the
projection centre \(\phi\).

Our first main theorem is the following.

\begin{theorem}[Relative Clifford dictionary]
\label{thm:intro-clifford}
Assume that \(Y_f/X\) admits a split relative projection centre
\(\phi\).  Then there is an equivalence of exact categories
\[
\Lambda_\phi:
\Ulr(Y_f/X)
\xrightarrow{\sim}
\Rep_X(C_{f,\phi})
\]
compatible with arbitrary base change.

If
\[
\Lambda_\phi(F)=(B,\rho),
\]
then
\[
B\simeq p_*F,
\qquad
\rk B=d\,\rk F.
\]
Equivalently,
\[
q_{\phi *}F\simeq\pi'^*B,
\]
where
\[
\pi':\mathbf P_X(V)\longrightarrow X.
\]

Moreover, the equivalence restricts to
\[
\Ulr_{\mathrm{nor}}(Y_f/X)
\simeq
\Rep_X^{\mathrm{triv}}(C_{f,\phi}),
\]
where the right-hand side consists of representations on trivial
coefficient bundles.
\end{theorem}

When \(X=\Spec k\), the coefficient bundle is only a vector space and
this recovers the usual projection criterion.  Over a
positive-dimensional base, the bundle \(p_*F\) is part of the
geometry.  For example, in the relative Veronese conic studied later,
relatively Ulrich bundles exist in every positive rank, while no
nonzero relatively Ulrich bundle is normalized when the twisting line
bundle on the base is nontrivial.

This is also the point where our terminology differs from some recent
work on finite covers.  For a fixed finite morphism
\[
\pi:Z\longrightarrow T,
\]
Mohan Kumar--Narayanan--Parameswaran, Parameswaran--Pine, and
Biswas--Pine study the existence of sheaves for which \(\pi_*E\) is
trivial \cite{MNP25,PP26,BP26}.  In our projection-centre setting,
this is exactly the normalized case.  Our category
\(\Ulr(Y_f/X)\) is larger because the coefficient bundle on \(X\) is
allowed to be arbitrary.

\subsection{Quadric fibrations: an intrinsic Clifford algebra}

For quadrics the theory becomes intrinsic.  Assume
\[
d=2,\qquad \operatorname{char}k\neq2.
\]
The equation
\[
f\in H^0(X,\Sym^2E^\vee\otimes L)
\]
is a line-bundle-valued quadratic form.  Let
\[
\mathscr C_0=\Cl_0(E,f,L)
\]
be its even Clifford algebra in the sense of Bichsel–Knus
\cite{BK94,Knu91}.  The algebra \(\mathscr C_0\) is locally free of
rank \(2^n\) and is defined globally without choosing a projection
centre.

\begin{theorem}[Centre-free Clifford dictionary and existence]
\label{thm:intro-quadric}
Let
\[
p:Y\longrightarrow X
\]
be a smooth quadric fibration of relative dimension \(n-1\ge1\).
Then
\[
\Ulr(Y/X)
\xrightarrow{\sim}
\Rep_X(\mathscr C_0)
\]
as exact categories, compatibly with arbitrary base change.  If \(F\)
corresponds to the \(\mathscr C_0\)-module \(B\), then
\[
B\simeq p_*F,
\qquad
\rk B=2\rk F.
\]

Consequently, the regular \(\mathscr C_0\)-module gives a relatively
Ulrich bundle of rank
\[
2^{n-1}.
\]
In particular,
\[
\uc(Y/X)\le2^{n-1}.
\]
No projection centre is required.
\end{theorem}

If a split projection centre exists, the two constructions agree.  In
that case an orthogonal splitting gives
\[
C_{f,\phi}\simeq\Cl(W,Q),
\]
and there is a canonical isomorphism
\[
\Cl(W,Q)\simeq\Cl_0(E,f,L).
\]
Under this isomorphism the projection-centre equivalence of
\cref{thm:intro-clifford} agrees with the intrinsic quadric
equivalence.  Thus the quadric theory extends the projection-centre
theory rather than defining a different category.

The intrinsic algebra also gives the exact relative complexity.  For a
Brauer class \(\gamma\), let \(r(\gamma)\) denote the smallest rank of
a nonzero locally free \(\gamma\)-twisted sheaf.

\begin{theorem}[Relative complexity of quadric fibrations]
\label{thm:intro-quadric-complexity}
Let \(Y/X\) be as above.  For every geometric point \(\bar x\to X\),
\[
\uc(Y_{\bar x})
=
2^{\lfloor (n-2)/2\rfloor}.
\]

If \(n\) is even, \(\mathscr C_0\) is Azumaya over \(X\).  Writing
\[
\alpha=[\mathscr C_0]\in\Br(X)[2],
\]
one has
\[
\\uc(Y/X)
=
\uc(Y_{\bar x})\,r(\alpha).
\]

If \(n\) is odd, the centre of \(\mathscr C_0\) defines the
discriminant double cover
\[
\tau:\widetilde X\longrightarrow X.
\]
If \(\widetilde X\) is connected and
\[
\beta\in\Br(\widetilde X)[2]
\]
is the Brauer class of \(\mathscr C_0\) over its centre, then
\[
\uc(Y/X)
=
2\,r(\beta)\,\uc(Y_{\bar x}).
\]

If
\[
\widetilde X\simeq X\sqcup X
\]
with component classes
\[
\beta_+,\beta_-\in\Br(X)[2],
\]
then
\[
\beta_+=\beta_-
\]
by \cite[Prop.~2.5]{Auel15}. Therefore
\[
\uc(Y/X)
=
\uc(Y_{\bar x})\,r(\beta_+)
=
\uc(Y_{\bar x})\,r(\beta_-).
\]
\end{theorem}

For a single smooth quadric, the minimal ranks are the ranks of the
spinor or half-spinor Ulrich bundles; see
\cite{BEH87,Ott88}.  In a family, the theorem shows two
different reasons for a rank jump.  A nontrivial Brauer class can
force every global Clifford module to have larger rank.  In the odd
case, the two fibrewise spinor types are parametrized by the
discriminant double cover; when this cover is connected, monodromy
exchanges the two types and forces an extra factor of \(2\).  The use
of even Clifford algebras, discriminant covers and Brauer classes in
the geometry of quadric fibrations is classical and appears, from
different viewpoints, in
\cite{Kuz08,ABB14,APS15}.  Our use of this structure is to
compute the minimal rank in the relative Ulrich category.

\subsection{Stability, moduli, and Clifford linearizations}

The projection-centre description also translates geometric
properties into representation theory.

\begin{theorem}[Stability and moduli]
\label{thm:intro-moduli}
Assume that a split relative projection centre is fixed.

For every geometric point \(\bar x\to X\), a relatively Ulrich bundle
\(F\) satisfies
\[
F_{\bar x}\text{ is Gieseker stable}
\quad\Longleftrightarrow\quad
\rho_{\bar x}\text{ is irreducible}.
\]

For fixed rank $r$, let $\mathcal U_r\to X$ be the stack whose
objects over a morphism $T\to X$ are relatively Ulrich bundles of
rank $r$ on
\[
Y_f\times_XT\longrightarrow T.
\]
Then
\[
\mathcal U_r\simeq [R_{dr}(f,\phi)/\GL_{dr}].
\]
Since $X$ is of finite type in Section~8, this stack is of finite
type over $X$.

If $\operatorname{char}k=0$, the stack has a good moduli space which
is affine over $X$, and the fibrewise stable locus is quasi-affine
over $X$.

This is a moduli problem over $X$. It should not be confused with
the moduli problem over $k$ of global vector bundles on the whole
family $Y_f$ which are relatively Ulrich over $X$.

In the absolute case $X=\Spec k$, still assuming
$\operatorname{char}k=0$, the Ulrich locus in the corresponding
Gieseker moduli space is open and affine.
\end{theorem}

We also separate two Clifford constructions which should not be
confused.  The algebra \(C_{f,\phi}\) is attached to a finite
projection and describes \emph{all} relatively Ulrich bundles once a
centre is chosen.  By contrast, the full generalized Clifford algebra
\(C_f\) of the original homogeneous form imposes the stronger power
identity
\[
M(y)^d=f(y)I.
\]
Representations of \(C_f\) therefore correspond only to
\emph{Clifford-linearized} relatively Ulrich bundles.  Equivalently,
they arise by restriction from relatively Ulrich bundles on the cyclic
cover
\[
\widetilde Y_f=\{w^d=f\}.
\]
The inclusion is in general strict.  Already for a smooth quadric
surface, an Ulrich line bundle of one ruling gives a linear matrix
factorization but does not give the required representation of the
full algebra \(C_f\).  The paper also gives a genuinely relative
constant-family version of this example.

\subsection{Examples and wildness from the base}

The general results are illustrated by three types of examples.

First, there are smooth quadric-surface bundles over curves for which
\[
\uc(Y_{\bar x})=1
\qquad\text{for every }\bar x,
\]
but
\[
\uc(Y/X)=2.
\]
Here the rank jump is caused by monodromy of the discriminant double
cover.

Second, there are conic bundles over abelian surfaces for which every
fibre has an Ulrich line bundle but no relative Ulrich line bundle
exists.  In these examples the obstruction is a nontrivial Brauer
class.

Third, for the relative Veronese conic
\[
Y=\mathbf P_X(\OO_X\oplus M^{-1}),
\qquad
\OO_Y(1)=\OO_{\mathbf P(\OO_X\oplus M^{-1})}(2),
\]
there is an equivalence
\[
\operatorname{VB}(X)
\xrightarrow{\sim}
\Ulr(Y/X),
\qquad
G\longmapsto
\OO_{\mathbf P(\OO_X\oplus M^{-1})}(1)\otimes p^*G.
\]
Hence
\[
\uc(Y/X)=1.
\]
If \(X\) is proper and connected and \(M\not\simeq\OO_X\), no
nonzero relatively Ulrich bundle is normalized.  Thus nontrivial
coefficient bundles are not a technical artefact; they are necessary
even in a very simple family.

The base can also make the category much larger.

\begin{theorem}[Wildness from the base]
\label{thm:intro-wildness}
Let \(X\) be a smooth projective curve of genus at least \(2\).
Assume that \(Y_f/X\) carries a relatively simple relatively Ulrich
bundle \(F\).  Then \(\Ulr(Y_f/X)\) is strictly wild.  More precisely,
for every finitely generated \(k\)-algebra \(\Lambda\) there is an
exact fully faithful functor
\[
\operatorname{mod}\Lambda
\longrightarrow
\Ulr(Y_f/X).
\]
\end{theorem}

The construction uses bundles of the form
\[
F\otimes p^*G,
\]
with \(G\) varying on the base.  On a geometric fibre this becomes
\[
F_{\bar x}^{\oplus\rk G}.
\]
Thus the parameters introduced by \(G\) disappear after restriction
to a fixed fibre.  This is different from the usual study of wild
ACM or Ulrich representation type on a fixed projective variety;
compare, for example, \cite{DrozdGreuel01,FaenziMalaspina17}.

\subsection{Relation to previous work}

We now place the results more precisely in the literature.

Roby and Childs introduced generalized Clifford algebras of
higher-degree forms, and their representation theory was developed
further by Van den Bergh, Haile, Kulkarni, and
Coskun--Kulkarni--Mustopa.

Krashen and Lieblich \cite{KL} construct a Clifford algebra for a
finite morphism in much greater generality. In particular,
\cite[Prop.~3.0.4]{KL} gives a base-change-compatible
representation--Ulrich correspondence on locally free coefficient
bundles. Their Appendix~A also treats explicit monic hypersurface
models; see especially \cite[Prop.~A.3.1]{KL} for the non-diagonal
monic case. Monic generalized Clifford algebras also appear in the matrix-pencil
construction of Pappacena \cite{Pappacena00} and in the work of
Chapman--Kuo \cite{ChapmanKuo15}. 

The contributions developed here are the global treatment of a
chosen relative projection centre, including the change of splitting,
the intrinsic centre-free Clifford description for quadric
fibrations, and the resulting global applications to relative Ulrich
rank, Brauer obstructions, monodromy, stability, moduli, and
wildness from the base.

For quadrics, line-bundle-valued Clifford algebras were developed by
Bichsel--Knus \cite{BK94}; see also Knus \cite{Knu91}.  Even Clifford
algebras play a central role in the derived and birational geometry of
quadric fibrations, notably in the work of Kuznetsov
\cite{Kuz08}, Auel--Bernardara--Bolognesi \cite{ABB14}, and
Auel--Parimala--Suresh \cite{APS15}.  Over a field, the relation
between Cohen--Macaulay modules on quadrics and Clifford modules is
classical; see Buchweitz--Eisenbud--Herzog \cite{BEH87}, and for
spinor bundles see Ottaviani \cite{Ott88}.  Our quadric results
use this Clifford structure in a relative setting to obtain an exact
formula for the minimal Ulrich rank, including the contribution of
the Brauer class and the monodromy of the discriminant cover.

Recent papers on finite coverings address a related but different
problem.  Mohan Kumar--Narayanan--Parameswaran prove existence results
for Ulrich bundles on smooth double covers of projective space
\cite{MNP25}.  Parameswaran--Pine prove existence on smooth cyclic
covers of projective spaces of arbitrary degree and study a
trivial-direct-image condition for finite morphisms \cite{PP26}.
Biswas--Pine study the same trivial-direct-image problem for more
general finite morphisms and give a detailed analysis for admissible
abelian Galois covers \cite{BP26}.  In our notation, these
trivial-direct-image objects correspond to the normalized part of the
projection-centre category.  Our main object is instead the full
relative category, in which
\[
q_{\phi *}F\simeq\pi'^*(p_*F)
\]
and \(p_*F\) may be nontrivial.

Accordingly, the main contribution of the present paper is not a new
abstract definition of a Clifford algebra for a finite morphism.
Rather, it is the combination of an explicit relative hypersurface
dictionary with the geometry of the base.  This makes it possible to
separate fibrewise existence from relative existence, to identify the
role of the coefficient bundle, to compare stability with
irreducibility, to describe moduli by representation spaces, to
compute the relative Ulrich complexity of quadric fibrations, and to
exhibit Brauer, monodromy and wildness phenomena which are invisible
on a single fibre.

\subsection{Organization of the paper}

Section~2 develops relatively Ulrich bundles, base change, and relative
linear presentations.  Section~3 studies relative projection centres,
the monic normal form, and the associated finite projection.
Section~4 constructs the relative monic Clifford algebra
\(C_{f,\phi}\), proves independence of the chosen splitting, and
establishes the Clifford dictionary and its normalized restriction.
Section~5 treats cyclic covers and Clifford-linearized Ulrich bundles,
and explains why the full generalized Clifford algebra does not
describe all Ulrich bundles on the branch hypersurface.

Section~6 studies existence.  It first gives the general
representation-theoretic existence criterion and then proves the
centre-free Clifford dictionary and existence theorem for smooth
quadric fibrations, together with the comparison
\[
C_{f,\phi}\simeq\Cl_0(E,f,L)
\]
when a centre exists.  Section~7 compares stability with
irreducibility, Section~8 studies the moduli stack and its good moduli
space, and Section~9 introduces the relative, normalized, and Clifford
complexity invariants.  Section~10 proves the wildness results.
Section~11 computes the relative Ulrich complexity of quadric
fibrations from the intrinsic even Clifford algebra and its Brauer and
monodromy data.  Section~12 gives the explicit cyclic, Veronese,
monodromy, and Brauer examples.  

\section{Relative Ulrich bundles}

\paragraph{Purpose of the section.}
We set up the relative notion, prove that it can be checked fibre by fibre, and derive the relative linear presentation that is later converted into a Clifford action.

\begin{definition}\label{def:relative-ulrich}
Let $p:Y\to X$ be flat and projective of pure relative dimension $m$, and let $\OO_Y(1)$ be relatively very ample.  A vector bundle $F$ on $Y$ is \emph{Ulrich relative to $p$} if
\[
Rp_*F(-j)=0\qquad(1\le j\le m).
\]
We write $\Ulr(Y/X)$ for the full subcategory of such bundles.
\end{definition}

\begin{proposition}[Base change]\label{prop:base-change}
Let $p:Y\to X$ be flat and projective and let $G$ be a coherent sheaf on $Y$ which is flat over $X$.  Then $Rp_*G$ is perfect locally on $X$ and its formation commutes with arbitrary base change.  In particular, for $x\in X$,
\[
Rp_*G\otimes_{\OO_X}^{\mathbf L}k(x)\simeq R\Gamma(Y_x,G_x).
\]
\end{proposition}

\begin{proof}
This is the standard theorem on cohomology and base change for a proper flat family with a flat coherent sheaf; one may use a finite complex of finite free modules representing cohomology locally on the base, as in \cite{Mum70,EGAIII}.  Tensoring that complex with any base algebra computes the cohomology after base change.
\end{proof}

\begin{proposition}[Fibrewise criterion]\label{prop:fibrewise-ulrich}
A vector bundle $F$ on $Y$ is relatively Ulrich if and only if $F_{\bar x}$ is Ulrich on $Y_{\bar x}$ for every geometric point $\bar x\to X$.
\end{proposition}

\begin{proof}
For $1\le j\le m$, put $K_j=Rp_*F(-j)$.  By \cref{prop:base-change}, the derived fibre of $K_j$ at $\bar x$ is $R\Gamma(Y_{\bar x},F_{\bar x}(-j))$.  If $K_j=0$, all fibre cohomology vanishes.  Conversely, if every derived fibre is zero, a local bounded finite free representative of $K_j$ becomes acyclic modulo every maximal ideal.  The derived form of Nakayama's lemma then gives $K_j=0$.
\end{proof}

\begin{corollary}\label{cor:ulrich-base-change}
Relative Ulrich bundles remain relatively Ulrich after arbitrary base change.
\end{corollary}

\subsection{Relative hypersurfaces and linear presentations}
Let $E$ be a vector bundle of rank $n+1$ and let
\[
Y_f=V(\widetilde f)\subset \PP_X(E),\qquad
\widetilde f\in H^0\bigl(X,\Sym^dE^\vee\otimes L\bigr),
\]
where the restriction of $\widetilde f$ to every projective fibre is nonzero.  Then $Y_f\to X$ is flat of pure relative dimension $n-1$; see \cite{EGAIV3}.  We write $i:Y_f\hookrightarrow\PP_X(E)$ and $p=\pi\circ i$.

We first recall the absolute calculation.  If $Z\subset\PP^n_K$ is a degree-$d$ hypersurface and $G$ is Ulrich of rank $r$, then \cite{ESW03,Bea18}
\begin{equation}\label{eq:absolute-linear}
0\longrightarrow \OO_{\PP^n}(-1)^{\oplus dr}
\longrightarrow \OO_{\PP^n}^{\oplus dr}
\longrightarrow i_*G\longrightarrow0.
\end{equation}
Hence
\[
P_G(t)=dr\binom{t+n-1}{n-1},
\]
and, in the convention of \cite[Def.~1.2.3]{HL10}, the reduced Hilbert polynomial is
\[
p_G(t)=\binom{t+n-1}{n-1}.
\]
This normalization is useful below because it is independent of $r$.

\begin{lemma}[Projection criterion]\label{lem:absolute-projection}
Let $o\in\PP^n(K)\setminus Z$ and let $\lambda:Z\to\PP^{n-1}_K$ be projection from $o$.  Then $\lambda$ is finite flat of degree $d$, $\lambda^*\OO(1)=\OO_Z(1)$, and
\[
G\text{ is Ulrich}\quad\Longleftrightarrow\quad
\lambda_*G\cong\OO_{\PP^{n-1}}^{\oplus dr}.
\]
\end{lemma}

\begin{proof}
After choosing coordinates with $o=(1:0:\cdots:0)$, the equation is monic of degree $d$ in the first coordinate, so the projection is finite flat of degree $d$.  Since finite pushforward is exact and preserves the relevant twists and cohomology, the Ulrich vanishing for $G$ is equivalent to the Ulrich vanishing for the rank-$dr$ bundle $\lambda_*G$ on projective space.  By \cite[Prop.~2.1]{ESW03}, the latter bundle is trivial.
\end{proof}

\begin{definition}
A vector bundle $F$ on $Y_f$ has a \emph{relative linear presentation} if there exist vector bundles $B_0,B_1$ on $X$ and an exact sequence
\begin{equation}\label{eq:relative-linear}
0\longrightarrow \pi^*B_1(-1)\xrightarrow{\alpha}\pi^*B_0
\longrightarrow i_*F\longrightarrow0.
\end{equation}
\end{definition}

\begin{proposition}[Relative linear presentation]\label{prop:relative-linear}
A vector bundle $F$ of rank $r$ on $Y_f$ is relatively Ulrich if and only if it has a relative linear presentation.  In that case
\[
B_0\cong p_*F,\qquad R^qp_*F=0\ (q>0),\qquad \rk B_0=\rk B_1=dr,
\]
and formation of $B_0$ commutes with base change.
\end{proposition}

\begin{proof}
Assume first that \eqref{eq:relative-linear} exists.  Twist by $\OO(-j)$ for $1\le j\le n-1$.  Since
\[
R\pi_*\OO(-j)=R\pi_*\OO(-j-1)=0
\qquad(1\le j<j+1\le n),
\]
the projection formula gives $Rp_*F(-j)=0$.

Conversely, assume $F$ is relatively Ulrich.  By \cref{prop:fibrewise-ulrich} and \eqref{eq:absolute-linear}, every fibre is globally generated, has no higher cohomology, and has $h^0=dr$.  Cohomology and base change therefore gives a locally free bundle
\[
B_0:=p_*F,
\qquad \rk B_0=dr,
\]
with $R^{>0}p_*F=0$.  The relative evaluation map $\pi^*B_0\to i_*F$ is surjective on every fibre, hence surjective by Nakayama.  Let $K$ be its kernel.  Fibrewise, $K_x\cong\OO_{\PP(E_x)}(-1)^{\oplus dr}$.  Since $K$ is flat over $X$ with locally free fibres, the fibrewise flatness criterion \cite{EGAIV3} makes $K$ locally free.  Now $K(1)$ is fibrewise trivial.  Thus
\[
B_1:=\pi_*K(1)
\]
is locally free of rank $dr$, and the adjunction $\pi^*B_1\to K(1)$ is an isomorphism.  This gives \eqref{eq:relative-linear}.  Applying $R\pi_*$ to the sequence recovers $B_0\cong p_*F$.
\end{proof}

\begin{definition}
A relatively Ulrich bundle $F$ of rank $r$ is \emph{normalized} if
\[
p_*F\cong\OO_X^{\oplus dr}.
\]
\end{definition}

\begin{example}\label{ex:coefficient-not-free}
If $Y=X\times Z$ with $Z\subset\PP^n$ a hypersurface, $G$ Ulrich on $Z$, and $M$ a line bundle on $X$, then
\[
F=\operatorname{pr}_2^*G\otimes p^*M
\]
is relatively Ulrich and
\[
p_*F\cong M^{\oplus d\rk G}.
\]
Thus even in a constant family the coefficient bundle need not be trivial.  This simple example is the reason that a relative Clifford correspondence must allow representations on arbitrary vector bundles, not only free modules.
\end{example}

\section{Relative projection centres}

\paragraph{Purpose of the section.}
We explain when projection from a section is defined, prove a sharp splitting criterion, put the equation in monic form, and obtain the finite flat projection that drives the rest of the paper.

\begin{definition}
A \emph{relative projection centre} is a line subbundle $\phi:S\hookrightarrow E$ such that the corresponding section $\sigma_\phi:X\to\PP_X(E)$ is disjoint from $Y_f$.  It is \emph{split} if $S$ is a direct summand of $E$.
\end{definition}

The restriction of $\widetilde f$ to the section is a section
\[
f(\phi)\in H^0(X,S^{-d}\otimes L).
\]

\begin{lemma}\label{lem:center-characterization}
The line subbundle $\phi$ is a projection centre if and only if $f(\phi)$ is nowhere vanishing.  In this case it gives a canonical isomorphism
\[
S^{\otimes d}\cong L.
\]
\end{lemma}

\begin{proof}
The intersection $\sigma_\phi(X)\cap Y_f$ is exactly the zero scheme of $\sigma_\phi^*\widetilde f=f(\phi)$.  A nowhere vanishing section trivializes $S^{-d}\otimes L$.
\end{proof}

\begin{proposition}[Sharp splitting criterion]\label{prop:split-centre}
If $d$ is invertible in $k$, every relative projection centre is split.
\end{proposition}

\begin{proof}
The homogeneous form $f$ has a first polar along $\phi$.  Locally, after trivializing $S$ and $L$, it is the linear map
\[
P:E\longrightarrow S^{-(d-1)}\otimes L
\]
whose value on $v$ is the coefficient of $t$ in $f(\phi(s)+tv)$, divided by $s^{d-1}$.  These local maps glue.  Euler's homogeneous identity gives
\[
P\circ\phi=d\,f(\phi):S\longrightarrow S^{-(d-1)}\otimes L.
\]
By \cref{lem:center-characterization}, $f(\phi)$ is a unit; since $d$ is invertible, $P\circ\phi$ is an isomorphism.  Hence
\[
(P\circ\phi)^{-1}\circ P:E\longrightarrow S
\]
is a retraction of $\phi$.
\end{proof}

\begin{proposition}[Local existence]\label{prop:center-local}
Every point of $X$ has an open neighbourhood on which $Y_f$ admits a split relative projection centre.
\end{proposition}

\begin{proof}
Let $x\in X$.  Since the fibre form $f_x$ is nonzero and $k(x)$ is infinite, choose $v_0\in E_x$ with $f_x(v_0)\ne0$.  On an affine neighbourhood where $E$ and $L$ are trivial, lift $v_0$ to a section $v$ which is part of a basis.  After shrinking, $f(v)$ is a unit.  Then $\OO v\subset E$ is a split line subbundle and, by \cref{lem:center-characterization}, a centre.
\end{proof}

\begin{remark}[A global obstruction]\label{rem:ample-no-centre}
If $X$ is projective of positive dimension and $Y_f$ is ample in $\PP_X(E)$, then no global centre exists.  Indeed, the complement of an ample effective divisor is affine, whereas the image $\sigma_\phi(X)$ of a projective positive-dimensional scheme cannot be a closed subscheme of an affine scheme.  Thus a global centre is a real hypothesis, even though centres exist Zariski-locally by \cref{prop:center-local}.
\end{remark}

\subsection{Monic normal form and finite projection}
Fix a split centre and a splitting
\[
E=S\oplus V.
\]
Put
\[
P'=\PP_X(V),\qquad \pi':P'\to X,
\qquad W=V\otimes S^{-1},
\qquad M=\OO_{P'}(1)\otimes\pi'^*S.
\]
Let $w$ be the coordinate corresponding to the projection $E\to S$.

\begin{lemma}[Monic normal form]\label{lem:monic-normal}
After identifying $L\cong S^{\otimes d}$ by $f(\phi)$, the equation can be written uniquely as
\[
\widetilde f=w^d+g_1w^{d-1}+\cdots+g_d,
\qquad g_j\in H^0(X,\Sym^jW^\vee)=H^0(P',M^j).
\]
Projection away from $\sigma_\phi(X)$ restricts to a finite flat map
\[
q:Y_f\longrightarrow P'
\]
of degree $d$, with $q^*\OO_{P'}(1)\cong\OO_{Y_f}(1)$.  Moreover
\[
q_*\OO_{Y_f}\cong \OO_{P'}\oplus M^{-1}\oplus\cdots\oplus M^{-(d-1)}
\]
as an $\OO_{P'}$-module, and its algebra structure is defined by the monic relation above.
\end{lemma}

\begin{proof}
The complement $\PP_X(E)\setminus\sigma_\phi(X)$ is the total space of $M$: a line not equal to $S$ is the graph of a homomorphism from a line in $V$ to $S$.  The splitting of $\Sym^dE^\vee\otimes L$ gives the coefficients $g_j$, and the coefficient of $w^d$ is $f(\phi)=1$.  Hence $Y_f$ is cut out in the total space of $M$ by a monic polynomial of degree $d$.  The division algorithm then gives the displayed free module with basis $1,w,\dots,w^{d-1}$, proving finiteness and flatness of degree $d$.
\end{proof}

\begin{lemma}[Local freeness from the projection]\label{lem:local-freeness}
Assume $p:Y_f\to X$ is smooth.  If $F$ is coherent on $Y_f$ and $q_*F$ is locally free on $P'$, then $F$ is locally free on $Y_f$.
\end{lemma}

\begin{proof}
First let $X=\Spec K$ be a field.  Both $Y_f$ and $P'$ are regular.  For a point $z\in P'$ and a point $y$ over it, the finite flat extension of regular local rings has equal dimension.  If $q_*F$ is free, the completed stalks of $F$ over the points above $z$ are direct summands of a free module over the completed local ring of $P'$.  The standard depth formula for finite local extensions gives full depth, so each stalk $F_y$ is maximal Cohen--Macaulay over the regular local ring $\OO_{Y_f,y}$.  Auslander--Buchsbaum implies that it is free.

For general $X$, the affineness of $q$ shows that $F$ is flat over $X$ because $q_*F$ is.  After base change to a point $x$, the pushforward $q_{x*}F_x$ is locally free; the field case makes $F_x$ locally free.  The fibrewise flatness criterion for the smooth morphism $Y_f\to X$ now implies that $F$ is locally free.
\end{proof}

\begin{proposition}[Relative projection criterion]\label{prop:relative-projection}
Let $F$ be a vector bundle of rank $r$ on $Y_f$.  Then
\[
F\in\Ulr(Y_f/X)
\quad\Longleftrightarrow\quad
q_*F\cong\pi'^*B
\]
for a vector bundle $B$ of rank $dr$ on $X$.  In that case $B\cong p_*F$.  Moreover $F$ is normalized if and only if $q_*F$ is trivial.
\end{proposition}

\begin{proof}
The bundle $q_*F$ has rank $dr$ and commutes with base change.  If $F$ is relatively Ulrich, \cref{lem:absolute-projection} gives $(q_*F)_x\cong\OO_{P'_x}^{\oplus dr}$ on every fibre.  Hence $B:=\pi'_*q_*F$ is locally free and the adjunction map $\pi'^*B\to q_*F$ is an isomorphism fibrewise, hence globally.  The converse follows by applying \cref{lem:absolute-projection} to each fibre and then \cref{prop:fibrewise-ulrich}.  Finally
\[
p_*F=\pi'_*q_*F\cong B.
\]
\end{proof}

\section{The relative monic Clifford algebra}

\paragraph{Purpose of the section.}
The finite projection turns a relatively Ulrich bundle into a pullback $\pi'^*B$ together with multiplication by $w$.  We now encode that multiplication as a representation and prove that no information is lost.

Choose locally a basis $e_1,\dots,e_n$ of $W$ with dual variables $y_1,\dots,y_n$.  For an $\OO_X$-linear map $\mu:W\to R$ into a possibly noncommutative $\OO_X$-algebra, put
\[
M_\mu(y)=\sum_i y_i\mu(e_i)
\]
and
\[
P_\mu(y)=M_\mu(y)^d+g_1(y)M_\mu(y)^{d-1}+\cdots+g_d(y).
\]
The variables $y_i$ are central.  This expression is independent of the chosen basis because it is obtained from the tautological tensor $\id_W\in W\otimes W^\vee$.

\begin{definition}\label{def:monic-clifford}
The \emph{relative monic Clifford algebra} is
\[
C_{f,\phi}=T_{\OO_X}(W)/I_{f,\phi},
\]
where $I_{f,\phi}$ is generated locally by all coefficients of $P_\iota(y)$, with $\iota:W\to T(W)$ the canonical inclusion.  Thus maps of algebras $C_{f,\phi}\to R$ are the same as $\OO_X$-linear maps $\mu:W\to R$ satisfying $P_\mu(y)=0$.
\end{definition}

Because $k$ is infinite, the coefficient identity $P_\mu(y)=0$ is equivalent to
\[
\mu(v)^d+g_1(v)\mu(v)^{d-1}+\cdots+g_d(v)=0
\]
for every local section $v$ of $W$.  When $g_1=\cdots=g_{d-1}=0$, this is the usual generalized Clifford relation $\mu(v)^d=-g_d(v)$.

\begin{lemma}[Independence of the splitting]
\label{lem:independence-splitting}
Let $\phi:S\hookrightarrow E$ be a split relative projection centre.
Let $s$ and $s'$ be two splittings of $E\to E/S$, and put
\[
W=(E/S)\otimes S^{-1}.
\]
Let $C_s$ and $C_{s'}$ be the monic Clifford algebras obtained from
the two splittings. Then there is a canonical isomorphism
\[
\theta:C_{s'}\xrightarrow{\sim}C_s.
\]
More precisely, the two splittings differ by a section
$\ell\in H^0(X,W^\vee)$, the corresponding coordinates satisfy
$w'=w+\ell$, and
\[
\theta(v)=v+\ell(v)\cdot 1 \qquad (v\in W).
\]
Under this isomorphism the two Clifford functors agree. Hence the
algebra depends only on $\phi$, up to canonical isomorphism.
\end{lemma}

\begin{proof}
Write
\[
P(w)=w^d+g_1w^{d-1}+\cdots+g_d,\qquad
P'(w')=(w')^d+g'_1(w')^{d-1}+\cdots+g'_d .
\]
Since the equations define the same hypersurface,
\[
P'(w+\ell(y))=P(w)
\]
in $\mathcal O_X[y,w]$. If $\mu:W\to R$ is an $\mathcal O_X$-linear
map to an $\mathcal O_X$-algebra $R$, set
$\mu'=\mu+\ell\cdot 1_R$. Substitution of $M_\mu(y)$ for $w$ shows
that $\mu$ satisfies the relations of $C_s$ if and only if $\mu'$
satisfies the relations of $C_{s'}$. The universal property gives
mutually inverse maps $C_{s'}\rightleftarrows C_s$. Compatibility of
the module structures follows because multiplication by $w'$ equals
multiplication by $w+\ell$.
\end{proof}

\subsection{Relation with the Clifford algebra of a finite morphism}

The construction below is closely related to the Clifford algebra of a
finite morphism introduced by Krashen and Lieblich \cite{KL}.  Recall
that the split projection centre determines a finite flat morphism
\[
q_\phi:Y_f\longrightarrow \mathbf P_X(V)
\]
of degree $d$.  Krashen--Lieblich associate a Clifford algebra to such
a finite morphism by a universal property involving algebra structures
on the pushforward $q_{\phi *}\OO_{Y_f}$; their construction is
compatible with base change \cite[Theorem~2.2.8]{KL}.

In the classical situation of a hypersurface
\[
x_0^d-f(x_1,\ldots,x_n)=0
\]
projected to $\mathbf P^{n-1}$ by forgetting the coordinate $x_0$,
Krashen--Lieblich identify the Clifford algebra of the finite
projection with the usual generalized Clifford algebra of $f$;
see \cite[Proposition~A.1.1]{KL}.  The algebra $C_{f,\phi}$ introduced
below is the corresponding explicit monic algebra adapted to the
relative hypersurface and the chosen projection centre.

\begin{definition}
For a sheaf of algebras $A$ on $X$, let $\Rep_X(A)$ be the category of pairs $(B,\rho)$ with $B$ a vector bundle on $X$ and $\rho:A\to\End(B)$ a unital homomorphism.  We use the exact structure induced from the category of $A$-modules.
\end{definition}

\begin{lemma}[Adjunction for the matrix coefficients]\label{lem:adjunction}
For vector bundles $B,B'$ and $j\ge0$,
\[
\Hom_{P'}(\pi'^*B,\pi'^*B'\otimes M^j)
\cong
H^0\bigl(X,\Hom(B,B')\otimes\Sym^jW^\vee\bigr).
\]
For $j=1$, a map $a:\pi'^*B\to\pi'^*B\otimes M$ corresponds to a map $\mu:W\to\End(B)$, and the $j$-fold iterate of $a$ corresponds to $M_\mu(y)^j$.
\end{lemma}

\begin{proof}
Use the projection formula and
\[
\pi'_*M^j=\Sym^jV^\vee\otimes S^j=\Sym^jW^\vee.
\]
The statement about iterates is checked after trivializing the bundles, where $a$ is exactly a matrix of linear forms $\sum y_i\mu(e_i)$.
\end{proof}

Let $F\in\Ulr(Y_f/X)$ and put $B=p_*F$.  By \cref{prop:relative-projection}, $q_*F\cong\pi'^*B$.  Multiplication by $w$ gives
\[
F\longrightarrow F\otimes q^*M,
\]
and after $q_*$ and the projection formula it gives
\[
a_F:\pi'^*B\longrightarrow\pi'^*B\otimes M.
\]
Let $\mu_F:W\to\End(B)$ be the corresponding map.

\begin{lemma}\label{lem:relation-action}
The map $\mu_F$ satisfies the monic identity, hence defines a representation
\[
\rho_F:C_{f,\phi}\longrightarrow\End(B).
\]
\end{lemma}

\begin{proof}
On $Y_f$, multiplication by the function $w$ satisfies its defining equation
\[
w^d+g_1w^{d-1}+\cdots+g_d=0.
\]
After pushing forward by $q$ and identifying $q_*F$ with $\pi'^*B$, \cref{lem:adjunction} turns this identity into $P_{\mu_F}(y)=0$.
\end{proof}

\begin{theorem}[Projection linearization]\label{thm:clifford-dictionary}
Assume $p:Y_f\to X$ is smooth and $\phi$ is a split relative projection centre.  The assignment
\[
\Lambda_\phi:\Ulr(Y_f/X)\longrightarrow\Rep_X(C_{f,\phi}),
\qquad F\longmapsto(p_*F,\rho_F),
\]
is an equivalence of exact $k$-linear categories.  It has the following properties.
\begin{enumerate}
\item If $\rk F=r$, then $\rk p_*F=dr$.  Hence every representation in the image, and therefore every object of $\Rep_X(C_{f,\phi})$, has coefficient rank divisible by $d$.
\item $F$ is normalized if and only if its coefficient bundle is trivial.
\item The equivalence is compatible with arbitrary base change $T\to X$.
\item For every vector bundle $G$ on $X$,
\[
\Lambda_\phi(F\otimes p^*G)
\cong (p_*F\otimes G,\rho_F\otimes\id_G).
\]
\end{enumerate}
\end{theorem}

\begin{proof}
We describe a quasi-inverse.  By \cref{lem:monic-normal}, the finite morphism $q$ is affine and
\[
A_f:=q_*\OO_{Y_f}
\]
is the algebra generated by the line bundle $M^{-1}$ with the monic relation.  Let $(B,\rho)$ be a representation of $C_{f,\phi}$.  Via \cref{lem:adjunction}, the action $\rho|_W$ gives a map
\[
a:\pi'^*B\longrightarrow\pi'^*B\otimes M.
\]
Equivalently, after twisting, it gives multiplication by the generator $w$ on the vector bundle $\pi'^*B$.  The monic identity is exactly the condition that this action factors through $A_f$.  Thus $\pi'^*B$ becomes a locally free $\OO_{P'}$-module equipped with an $A_f$-module structure.

Since $q$ is affine, $A_f$-modules on $P'$ are the same as quasi-coherent sheaves on $Y_f$.  Let $F$ be the sheaf corresponding to $\pi'^*B$.  Then $q_*F\cong\pi'^*B$ is locally free.  By \cref{lem:local-freeness}, $F$ is locally free, and by \cref{prop:relative-projection} it is relatively Ulrich.  This construction is inverse to $F\mapsto(p_*F,\rho_F)$ because the $A_f$-module structure is determined by multiplication by $w$.

Morphisms are handled in the same way: an intertwiner of Clifford actions is precisely an $A_f$-linear map after pullback to $P'$, hence a morphism of the corresponding sheaves on $Y_f$.  The constructions preserve short exact sequences, so the equivalence is exact.

The rank and normalization statements follow from \cref{prop:relative-projection}.  Base change is compatible with the monic equation, with the algebra $A_f$, and with $p_*F$ by \cref{prop:relative-linear,cor:ulrich-base-change}.  Finally the projection formula gives
\[
p_*(F\otimes p^*G)\cong p_*F\otimes G,
\]
and multiplication by $w$ acts as $\rho_F\otimes\id_G$.
\end{proof}
\begin{remark}[Relation with the finite-morphism correspondence]
\label{rem:KL-comparison}
The relation between representations of a Clifford algebra and Ulrich
sheaves for a finite morphism was developed in considerable generality
by Krashen and Lieblich \cite[Proposition~3.0.4]{KL}.  In their
framework, if
\[
\varphi:Z\longrightarrow T
\]
is a finite morphism of constant degree $d$ satisfying their condition
$(C^\star)$, then representations of the Clifford algebra of
$\varphi$ correspond to Ulrich sheaves for $\varphi$, and a rank-$r$
Ulrich sheaf corresponds to a representation of rank $dr$.

For the finite projection
\[
q_\phi:Y_f\longrightarrow \mathbf P_X(V)
\]
considered here, the preceding theorem gives this correspondence in an
explicit hypersurface form.  The choice of a split relative projection
centre produces the global monic equation
\[
w^d+g_1w^{d-1}+\cdots+g_d=0,
\]
and hence an explicit algebra $C_{f,\phi}$ generated by the rank-$n$
bundle
\[
W=V\otimes S^{-1}.
\]
Moreover, the coefficient bundle is identified canonically as
\[
B\simeq p_*F.
\]
This concrete realization is the form used in the later study of
stability, moduli, complexity, monodromy, and Brauer obstructions.
\end{remark}

\begin{definition}[Trivial-coefficient representations]
\label{def:trivial-coefficient-representations}
Let
\[
\Rep_X^{\mathrm{triv}}(C_{f,\phi})
\subset
\Rep_X(C_{f,\phi})
\]
be the full subcategory consisting of representations $(B,\rho)$
whose underlying coefficient bundle is trivial:
\[
B\simeq \OO_X^{\oplus N}
\]
for some $N\geq 1$.

Let
\[
\Ulr_{\mathrm{nor}}(Y_f/X)
\subset
\Ulr(Y_f/X)
\]
be the full subcategory of normalized relatively Ulrich bundles.
Thus, if $F$ has rank $r$, then
\[
F\in\Ulr_{\mathrm{nor}}(Y_f/X)
\quad\Longleftrightarrow\quad
p_*F\simeq\OO_X^{\oplus dr}.
\]
\end{definition}

\begin{corollary}[Normalized Clifford correspondence]
\label{cor:normalized-clifford}
Assume that $p:Y_f\to X$ is smooth and that $\phi$ is a split
relative projection centre.  Then the equivalence
\[
\Lambda_\phi:
\Ulr(Y_f/X)
\xrightarrow{\sim}
\Rep_X(C_{f,\phi})
\]
of \cref{thm:clifford-dictionary} restricts to an equivalence of
full subcategories
\[
\Ulr_{\mathrm{nor}}(Y_f/X)
\xrightarrow{\sim}
\Rep_X^{\mathrm{triv}}(C_{f,\phi}).
\]

More precisely, for every $r\geq1$, normalized relatively Ulrich
bundles of rank $r$ correspond exactly to representations
\[
\rho:
C_{f,\phi}
\longrightarrow
\End_{\OO_X}\!\left(\OO_X^{\oplus dr}\right)
\simeq
\Mat_{dr}(\OO_X).
\]
\end{corollary}

\begin{proof}
Let $F\in\Ulr(Y_f/X)$ and write
\[
\Lambda_\phi(F)=(B,\rho_F).
\]
By \cref{thm:clifford-dictionary},
\[
B\simeq p_*F
\qquad\text{and}\qquad
\rk B=d\,\rk F.
\]
Hence $F$ is normalized if and only if
\[
B\simeq\OO_X^{\oplus d\rk F},
\]
which is precisely the condition that $\Lambda_\phi(F)$ has trivial
coefficient bundle.

Conversely, let $(B,\rho)$ be a representation of $C_{f,\phi}$ with
\[
B\simeq\OO_X^{\oplus dr}.
\]
The quasi-inverse in \cref{thm:clifford-dictionary} produces a
relatively Ulrich bundle $F$ of rank $r$ satisfying
\[
p_*F\simeq B\simeq\OO_X^{\oplus dr}.
\]
Thus $F$ is normalized.

Full faithfulness and essential surjectivity follow by restricting the
equivalence of \cref{thm:clifford-dictionary}.
\end{proof}
\begin{remark}
\label{rem:normalized-not-exact}
The normalized subcategory should not in general be regarded as an
extension-closed exact subcategory of $\Ulr(Y_f/X)$.

Indeed, under the Clifford correspondence an extension of two
representations with trivial coefficient bundles may have a nontrivial
coefficient bundle on $X$.  Equivalently, an extension
\[
0\longrightarrow
\OO_X^{\oplus a}
\longrightarrow B
\longrightarrow
\OO_X^{\oplus b}
\longrightarrow0
\]
need not split when $H^1(X,\OO_X)\neq0$.

Thus \cref{cor:normalized-clifford} is an equivalence of full
subcategories, while the full category
$\Ulr(Y_f/X)$ is the natural exact category.
\end{remark}
\begin{corollary}
\label{cor:normalized-matrix-representations}
Under the hypotheses of
\cref{cor:normalized-clifford}, choosing a trivialization
\[
p_*F\simeq\OO_X^{\oplus dr}
\]
identifies a normalized relatively Ulrich bundle of rank $r$ with a
matrix representation
\[
\rho_F:
C_{f,\phi}\longrightarrow
\Mat_{dr}(\OO_X).
\]

Two such trivialized representations determine isomorphic normalized
relatively Ulrich bundles if and only if they are isomorphic as
$C_{f,\phi}$-modules.  After fixing the trivializations, this means
that the representations differ by conjugation by an automorphism of
$\OO_X^{\oplus dr}$.
\end{corollary}

\begin{remark}[Why the projection-centre algebra is necessary]
\label{rem:projection-centre-vs-full}
The algebra $C_{f,\phi}$ should not be confused with the ordinary
full generalized Clifford algebra of the homogeneous form.

The projection-centre algebra is defined in the line-bundle-valued
setting used throughout the paper. By contrast, the ordinary full
algebra with relation
\[
M(y)^d=f(y)I
\]
is used here only when
\[
L\simeq\OO_X,
\]
so that $f$ is scalar-valued.

Under this scalar-valued hypothesis, a normalized relatively Ulrich
bundle need not determine a representation of the full algebra. A
linear Ulrich presentation gives a matrix factorization, but it does
not in general imply the stronger identity
\[
M(y)^d=f(y)I.
\]
The latter is an additional Clifford-linearization condition. The
cyclic-cover comparison in Section~5 further assumes that $d$ is
invertible in $k$.
\end{remark}

\begin{corollary}[Relative cyclic covers]\label{cor:relative-cyclic}
Let $S\in\Pic(X)$, let $W$ be a vector bundle of rank $n$, let $g\in H^0(X,\Sym^dW^\vee)$, and set
\[
E=S\oplus(W\otimes S),\qquad L=S^{\otimes d}.
\]
Let $Y_g\subset\PP_X(E)$ be defined by $w^d-g=0$.  If $Y_g\to X$ is smooth, then
\[
\Ulr(Y_g/X)\simeq\Rep_X(C_g),
\]
where $C_g$ is generated by $W$ with relation $M(y)^d=g(y)$.  This statement does not require a characteristic hypothesis.
\end{corollary}

\begin{remark}
The absence of a characteristic hypothesis in \cref{cor:relative-cyclic} is important.  The centre is visibly split from the direct-sum presentation of $E$, so the proof of \cref{thm:clifford-dictionary} does not need \cref{prop:split-centre}.
\end{remark}

\section{Cyclic covers and Clifford-linearized bundles}

\paragraph{Purpose of the section.}
There are two Clifford algebras in the story: the monic algebra attached to a projection centre and the full generalized Clifford algebra of a homogeneous form.  We now relate the second one to a cyclic cover and explain exactly which Ulrich bundles on the branch hypersurface carry a power factorization.

Assume in this section that $L\cong\OO_X$, $d\ge2$, $d$ is invertible in $k$, and $Y_f\to X$ is smooth.  Let $C_f$ be the generalized Clifford algebra generated by $E$ with relation
\[
M(y)^d=f(y)I.
\]

\begin{definition}
A \emph{linear power factorization} of $\widetilde f$ is a vector bundle $B$ on $X$ together with
\[
\alpha:\pi^*B(-1)\longrightarrow\pi^*B
\]
such that
\[
\alpha^{[d]}:=\alpha\circ\alpha(-1)\circ\cdots\circ\alpha(-(d-1))
=\widetilde f\,\id.
\]
We denote the category by $\MFpow_X(f)$.

A Clifford-linearized relatively Ulrich bundle is a triple
\[
(F,B,\alpha)
\]
such that
\[
0\longrightarrow
\pi^*B(-1)
\xrightarrow{\alpha}
\pi^*B
\longrightarrow
i_*F
\longrightarrow0
\]
is exact and
\[
\alpha^{[d]}=\widetilde f\,\id.
\]

A morphism
\[
(F,B,\alpha)\longrightarrow(F',B',\alpha')
\]
is a morphism
\[
u:B\longrightarrow B'
\]
such that
\[
\pi^*u\circ\alpha
=
\alpha'\circ\pi^*u(-1).
\]
It induces a morphism $F\to F'$. We denote this category by
\[
\UlrCl(Y_f/X).
\]
\end{definition}

Let
\[
\widetilde Y_f=\{w^d=f\}\subset\PP_X(\OO_X\oplus E)
\]
and let $\widetilde p:\widetilde Y_f\to X$.  The summand $\OO_X\subset\OO_X\oplus E$ is a split centre.  Let $R=\widetilde Y_f\cap\{w=0\}$.
\begin{remark}
The forgetful functor
\[
\UlrCl(Y_f/X)\longrightarrow\Ulr(Y_f/X)
\]
is not asserted to be full. A morphism of the underlying relatively
Ulrich bundles need not preserve the chosen Clifford linearization.
\end{remark}

\begin{lemma}\label{lem:cyclic-cover}
The morphism $\widetilde Y_f\to X$ is smooth of relative dimension $n$ and has geometrically integral fibres.  Moreover $R\cong Y_f$, and
\[
0\longrightarrow\OO_{\widetilde Y_f}(-1)
\xrightarrow{\,w\,}\OO_{\widetilde Y_f}
\longrightarrow\OO_R\longrightarrow0
\]
is exact.
\end{lemma}

\begin{proof}
On a geometric fibre a singular point of $w^d-f=0$ must satisfy $w=0$ because $d$ is invertible, and then it would be a singular point of $f=0$.  Thus the cyclic cover is smooth.  A smooth positive-dimensional hypersurface in projective space is connected, and a connected regular scheme has only one irreducible component, so the fibres are integral.  The identification $R\cong Y_f$ is immediate from $w=0$, and the displayed sequence follows because $w$ is a non-zero-divisor.
\end{proof}

\begin{theorem}[Cyclic-cover equivalence]\label{thm:cyclic-equivalence}
There are equivalences of exact categories, compatible with arbitrary base change,
\[
\Ulr(\widetilde Y_f/X)
\simeq \Rep_X(C_f)
\simeq \MFpow_X(f)
\simeq \UlrCl(Y_f/X).
\]
Under the composite, a relatively Ulrich bundle $\widetilde F$ on $\widetilde Y_f$ is sent to
\[
\bigl(\widetilde F|_R,\,\widetilde p_*\widetilde F,\,\text{multiplication by }w\bigr).
\]
Consequently $F\in\Ulr(Y_f/X)$ is Clifford-linearizable if and only if
\[
F\cong\widetilde F|_R
\]
for some $\widetilde F\in\Ulr(\widetilde Y_f/X)$.
\end{theorem}

\begin{proof}
The first equivalence is \cref{thm:clifford-dictionary} applied to the split centre of the cyclic cover; its monic Clifford algebra is $C_f$.  By \cref{lem:adjunction}, a representation $\mu:E\to\End(B)$ corresponds to a map $\alpha:\pi^*B(-1)\to\pi^*B$, and the identity $M_\mu(y)^d=f(y)I$ is exactly $\alpha^{[d]}=\widetilde f\,I$.  This gives $\Rep_X(C_f)\simeq\MFpow_X(f)$.

Now let $\widetilde F$ correspond to $(B,\alpha)$.  Tensor the exact sequence in \cref{lem:cyclic-cover} by $\widetilde F$ and push forward along the finite projection $\widetilde q:\widetilde Y_f\to\PP_X(E)$.  Since $\widetilde q_*\widetilde F\cong\pi^*B$, multiplication by $w$ becomes $\alpha$, and we obtain
\[
0\longrightarrow\pi^*B(-1)\xrightarrow{\alpha}\pi^*B
\longrightarrow i_*(\widetilde F|_R)\longrightarrow0.
\]
This produces a Clifford-linearized relatively Ulrich bundle.  Conversely a triple is determined by its power factorization, so this construction is essentially surjective and fully faithful.  Exactness and base-change compatibility follow from the corresponding statements in \cref{thm:clifford-dictionary} and from exactness of finite pushforward.
\end{proof}

\begin{corollary}[Normalized full-Clifford correspondence]
\label{cor:normalized-full-clifford}
In the setting of \cref{thm:cyclic-equivalence}, let
\[
\Rep_X^{\mathrm{triv}}(C_f)
\]
denote the full subcategory of representations of $C_f$ whose
underlying coefficient bundle is trivial.

Let
\[
\UlrCl_{\mathrm{nor}}(Y_f/X)
\]
denote the full subcategory of Clifford-linearized relatively Ulrich
bundles $(F,B,\alpha)$ for which $B$ is trivial.

Then
\[
\Ulr_{\mathrm{nor}}(\widetilde Y_f/X)
\simeq
\Rep_X^{\mathrm{triv}}(C_f)
\simeq
\UlrCl_{\mathrm{nor}}(Y_f/X).
\]

In particular, a normalized relatively Ulrich bundle on the branch
hypersurface $Y_f$ corresponds to a representation of $C_f$ only when
it carries the additional Clifford-linearization.
\end{corollary}

\begin{proof}
Restrict the equivalences of
\cref{thm:cyclic-equivalence} to objects whose coefficient bundle $B$
is trivial.  Under the first equivalence,
\[
B\simeq\widetilde p_*\widetilde F,
\]
so triviality of $B$ is exactly normalization of $\widetilde F$.
Under the second equivalence the same bundle $B$ is the coefficient
bundle of the power factorization and of the Clifford-linearized
Ulrich bundle.
\end{proof}

\begin{proposition}[A normalized relative Ulrich bundle need not be
full-Clifford linearized]
\label{prop:relative-quadric-not-full-clifford}

Assume that $\operatorname{char}k\neq 2$, and let $X$ be a smooth
connected projective $k$-scheme.  Let
\[
Q=
V(x_0x_3-x_1x_2)
\subset \mathbf P^3_k.
\]
Then $Q$ is a smooth quadric surface, and under the identification
\[
Q\simeq \mathbf P^1\times\mathbf P^1
\]
its hyperplane bundle is
\[
\OO_Q(1)\simeq\OO_Q(1,1).
\]

Consider the constant family
\[
p:Y:=X\times Q\longrightarrow X
\]
inside
\[
\mathbf P_X(\OO_X^{\oplus4})
\simeq X\times\mathbf P^3.
\]
Let
\[
\operatorname{pr}_Q:Y=X\times Q\longrightarrow Q
\]
be the second projection and set
\[
F:=
\operatorname{pr}_Q^*\OO_Q(1,0).
\]

Then:

\begin{enumerate}
\item $F$ is a relatively Ulrich line bundle on $Y/X$;

\item $F$ is normalized:
\[
p_*F\simeq\OO_X^{\oplus2};
\]

\item nevertheless, $F$ does not arise from a rank-$2$
representation of the generalized Clifford algebra $C_f$ of the
original quadratic form
\[
f=x_0x_3-x_1x_2.
\]
In fact, there is no $\OO_X$-algebra homomorphism
\[
C_f\longrightarrow\Mat_2(\OO_X).
\]
\end{enumerate}
\end{proposition}

\begin{proof}
For every geometric point $\bar x\to X$, one has
\[
F_{\bar x}\simeq\OO_Q(1,0).
\]
We first recall that $\OO_Q(1,0)$ is Ulrich with respect to
$\OO_Q(1,1)$.

Indeed,
\[
\OO_Q(1,0)(-1)
=
\OO_Q(0,-1)
\]
and
\[
\OO_Q(1,0)(-2)
=
\OO_Q(-1,-2).
\]
By the K\"unneth formula and the standard cohomology of line bundles
on $\mathbf P^1$,
\[
H^i(Q,\OO_Q(0,-1))=0
\]
and
\[
H^i(Q,\OO_Q(-1,-2))=0
\]
for every $i$.  Since $\dim Q=2$, this proves that
$\OO_Q(1,0)$ is Ulrich.  Hence $F$ is relatively Ulrich.

Next, by cohomology and base change,
\[
p_*F
\simeq
H^0(Q,\OO_Q(1,0))\otimes_k\OO_X.
\]
Since
\[
h^0(Q,\OO_Q(1,0))=2,
\]
we obtain
\[
p_*F\simeq\OO_X^{\oplus2}.
\]
Thus $F$ is normalized.

It remains to prove that $F$ does not come from a rank-$2$
representation of the full generalized Clifford algebra $C_f$.

Suppose, to the contrary, that there exists an $\OO_X$-algebra
homomorphism
\[
\rho:C_f\longrightarrow\Mat_2(\OO_X).
\]
Restricting to any geometric point $\bar x\to X$ gives a
$2$-dimensional representation
\[
\rho_{\bar x}:
(C_f)_{\bar x}
\longrightarrow
\Mat_2(k).
\]
Let $e_0,e_1,e_2,e_3$ be the standard basis and put
\[
A_i:=\rho_{\bar x}(e_i).
\]
The defining relation of $C_f$ gives
\[
M(u)^2
=
(u_0u_3-u_1u_2)I_2,
\]
where
\[
M(u):=
u_0A_0+u_1A_1+u_2A_2+u_3A_3.
\]

Let
\[
K=k(u_0,u_1,u_2,u_3).
\]
The quadratic form
\[
q=u_0u_3-u_1u_2
\]
is not a square in $K$.  Therefore the polynomial
\[
T^2-q
\]
is irreducible over $K$.  Since
\[
M(u)^2=qI_2,
\]
the minimal polynomial of $M(u)$ over $K$ is $T^2-q$.  As $M(u)$ is a
$2\times2$ matrix, its characteristic polynomial is therefore also
\[
T^2-q.
\]
In particular,
\[
\operatorname{tr}M(u)=0.
\]

Hence we may write
\[
M(u)=
\begin{pmatrix}
a & b\\
c & -a
\end{pmatrix},
\]
where $a,b,c$ are linear forms in
$u_0,u_1,u_2,u_3$.  Squaring gives
\[
M(u)^2
=
(a^2+bc)I_2.
\]
Consequently
\[
u_0u_3-u_1u_2=a^2+bc.
\]

The quadratic form on the right factors through the vector space
spanned by the three linear forms $a,b,c$, and hence has rank at most
$3$.  On the other hand,
\[
u_0u_3-u_1u_2
\]
is a nondegenerate quadratic form of rank $4$.  This is a
contradiction.

Thus no rank-$2$ representation of $C_f$ exists, and in particular
the normalized relatively Ulrich bundle $F$ does not arise from a
representation of the full generalized Clifford algebra $C_f$.
\end{proof}

\begin{remark}[Why the projection-centre algebra gives the correct correspondence]
\label{rem:relative-quadric-projection-centre}

The preceding proposition does not contradict the Clifford dictionary
of \cref{thm:clifford-dictionary}.  The latter uses the
projection-centre algebra $C_{f,\phi}$ rather than the full generalized
Clifford algebra $C_f$.

Indeed, choose a point
\[
a\in\mathbf P^3_k\setminus Q.
\]
It defines a constant relative projection centre
\[
\phi:X\longrightarrow X\times\mathbf P^3
\]
which is disjoint from $Y=X\times Q$.  Since the ambient vector bundle
$\OO_X^{\oplus4}$ is trivial, this centre is split.

Projection from $\phi$ gives a finite flat morphism
\[
q_\phi:Y\longrightarrow\mathbf P_X^2
\]
of degree $2$.  Since $F$ is relatively Ulrich and
\[
p_*F\simeq\OO_X^{\oplus2},
\]
the projection criterion gives
\[
q_{\phi *}F
\simeq
\OO_{\mathbf P_X^2}^{\oplus2}.
\]
Hence the corrected Clifford correspondence produces a representation
\[
C_{f,\phi}
\longrightarrow
\Mat_2(\OO_X).
\]

Thus the same bundle $F$ gives a rank-$2$ representation of
$C_{f,\phi}$ but no rank-$2$ representation of $C_f$.  This shows
concretely that the two Clifford algebras play different roles:
\[
C_{f,\phi}
\quad\text{describes all relatively Ulrich bundles,}
\]
whereas
\[
C_f
\quad\text{describes only those carrying the additional
Clifford-linearization.}
\]
\end{remark}

\section{Existence of relatively Ulrich bundles}
\label{sec:relative-existence}

The Clifford correspondence of the previous section translates the
existence problem for relatively Ulrich bundles into a representation
problem on the base.  
In the relative setting, the existence of Ulrich bundles on every
geometric fibre does not by itself imply that bundles of the same rank
glue to a relatively Ulrich bundle on the total space.  We first give
the precise representation-theoretic existence criterion. We then prove a centre-free Clifford correspondence and an
existence theorem for smooth quadric fibrations in characteristic
different from \(2\).

\begin{proposition}[Existence criterion]
\label{prop:relative-existence-criterion}
Assume that
\[
p:Y_f\longrightarrow X
\]
is a smooth relative hypersurface of degree $d$ and that
\[
\phi:S\hookrightarrow E
\]
is a split relative projection centre.  Then the following are
equivalent:
\begin{enumerate}
\item there exists a relatively Ulrich bundle of positive rank on
      $Y_f/X$;
\item the relative monic Clifford algebra $C_{f,\phi}$ admits a
      representation
      \[
      \rho:C_{f,\phi}\longrightarrow \End(B)
      \]
      on a vector bundle $B$ of positive rank on $X$.
\end{enumerate}

More precisely, if $B$ has rank $N$, then
\[
d\mid N,
\]
and the relatively Ulrich bundle corresponding to $(B,\rho)$ has rank
\[
\frac{N}{d}.
\]
\end{proposition}

\begin{proof}
By \cref{thm:clifford-dictionary}, there is an equivalence of exact
categories
\[
\Lambda_\phi:
\Ulr(Y_f/X)
\xrightarrow{\sim}
\Rep_X(C_{f,\phi}),
\qquad
F\longmapsto (p_*F,\rho_F).
\]
Hence a relatively Ulrich bundle of positive rank exists if and only if
there is a representation of $C_{f,\phi}$ on a vector bundle of
positive rank.

If $F$ has rank $r$, the rank formula in
\cref{thm:clifford-dictionary} gives
\[
\rk(p_*F)=dr.
\]
Thus, if $(B,\rho)$ corresponds to $F$, then
\[
\rk B=d\,\rk F.
\]
Consequently
\[
d\mid\rk B
\qquad\text{and}\qquad
\rk F=\frac{\rk B}{d}.
\]
\end{proof}

\begin{corollary}[The regular-module construction]
\label{cor:regular-module-existence}
Assume the hypotheses of
\cref{prop:relative-existence-criterion}, and suppose moreover that
$X$ is connected and that $C_{f,\phi}$ is finite locally free as an
$\OO_X$-module of rank $N$.

Then $Y_f/X$ admits a relatively Ulrich bundle
\[
F_{\mathrm{reg}}
\]
of positive rank.  More precisely,
\[
d\mid N
\]
and
\[
\rk F_{\mathrm{reg}}=\frac{N}{d}.
\]
\end{corollary}

\begin{proof}
The algebra $C_{f,\phi}$ acts on itself by left multiplication:
\[
C_{f,\phi}
\longrightarrow
\End_{\OO_X}(C_{f,\phi}),
\qquad
a\longmapsto (b\mapsto ab).
\]
Since $C_{f,\phi}$ is finite locally free over $X$, its underlying
$\OO_X$-module is a vector bundle of rank $N$.  Therefore its left
regular module is an object of
\[
\Rep_X(C_{f,\phi}).
\]
Applying \cref{prop:relative-existence-criterion} gives a relatively
Ulrich bundle $F_{\mathrm{reg}}$ satisfying
\[
\rk C_{f,\phi}
=
d\,\rk F_{\mathrm{reg}}.
\]
Hence
\[
d\mid N
\qquad\text{and}\qquad
\rk F_{\mathrm{reg}}=\frac{N}{d}.
\]
\end{proof}

\begin{remark}[Why the regular representation does not always give existence]
\label{rem:regular-module-warning}
The finiteness hypothesis in
\cref{cor:regular-module-existence} is essential.

Krashen and Lieblich point out that, although the Clifford algebra of a
finite morphism may exist by a universal property, it need not in
general be locally free over the base; see
\cite[Theorem~2.2.8 and Remark~2.2.9]{KL}.  The same issue occurs for
generalized Clifford algebras associated with forms of degree
$d\geq3$: the definition of $C_{f,\phi}$ does not by itself imply that
it is finite locally free as an $\OO_X$-module.

Consequently, in arbitrary degree the left regular module of
$C_{f,\phi}$ need not be an object of
\[
\Rep_X(C_{f,\phi}),
\]
because its underlying $\OO_X$-module may fail to have finite rank.
Thus the Clifford correspondence reduces the existence problem to the
existence of a positive-rank finite locally free representation of
$C_{f,\phi}$, rather than proving existence automatically.

For relative quadrics the situation is different: the associated
ordinary Clifford algebra is finite locally free of rank $2^n$, and
hence its regular representation does produce a relatively Ulrich
bundle; see \cref{thm:quadric-existence}.
\end{remark}

\subsection{Quadric fibrations without a projection centre}

Assume throughout this subsection that
\[
\operatorname{char}k\neq 2.
\]
Let $E$ be a vector bundle of rank $n+1$ on $X$, let $L$ be a line
bundle on $X$, and let
\[
f\in H^0\!\left(X,\operatorname{Sym}^2E^\vee\otimes L\right)
\]
define a smooth quadric fibration
\[
p:Y=V(f)\longrightarrow X
\]
of relative dimension $n-1\geq1$.

Let
\[
\mathscr C_0:=\operatorname{Cl}_0(E,f,L)
\]
be the even Clifford algebra of the line-bundle-valued quadratic form
$(E,f,L)$ in the sense of Bichsel--Knus.  It is locally free of rank
\[
2^n
\]
as an $\OO_X$-module, and its formation commutes with arbitrary base
change.

\begin{theorem}[Centre-free Clifford dictionary for quadrics]
\label{thm:quadric-clifford}
There is an equivalence of exact categories
\[\operatorname{Ulr}(Y/X)
\xrightarrow{\sim}
\operatorname{Rep}_X(\mathscr C_0)
\]
compatible with arbitrary base change.

If a relatively Ulrich bundle $F$ corresponds to a locally free
$\mathscr C_0$-module $B$, then
\[
B\simeq p_*F
\]
and
\[
\rk B=2\,\rk F.
\]
\end{theorem}

\begin{proof}
Let $F$ be a relatively Ulrich bundle on $Y$.

By \cref{prop:relative-linear}, $F$ has a
relative linear presentation
\[
0\longrightarrow
\pi^*B_1(-1)
\xrightarrow{\alpha}
\pi^*B_0
\longrightarrow
i_*F
\longrightarrow0,
\]
where
\[
B_0\simeq p_*F
\]
and
\[
\rk B_0=\rk B_1=2\rk F.
\]
Because $i_*F$ is supported on the quadric $Y=V(f)$, multiplication by
$f$ annihilates $i_*F$. Hence the multiplication map
\[
f:
\pi^*B_0
\longrightarrow
\pi^*B_0(2)\otimes\pi^*L
\]
factors uniquely through
\[
\alpha(2)\otimes\id_{\pi^*L}:
\pi^*B_1(1)\otimes\pi^*L
\longrightarrow
\pi^*B_0(2)\otimes\pi^*L.
\]
Therefore there is a unique morphism
\[
\beta:
\pi^*B_0
\longrightarrow
\pi^*B_1(1)\otimes\pi^*L
\]
such that
\[
(\alpha(2)\otimes\id_{\pi^*L})\circ\beta
=
f\,\id_{\pi^*B_0}.
\]

Moreover,
\[
(\alpha(2)\otimes\id_{\pi^*L})
\circ
(\beta\circ\alpha)
=
f\,\alpha.
\]
Since
\[
\alpha(2)\otimes\id_{\pi^*L}
\]
is injective, we obtain
\[
\beta\circ\alpha
=
f\,\id_{\pi^*B_1(-1)}.
\]
Here both sides are maps
\[
\pi^*B_1(-1)
\longrightarrow
\pi^*B_1(1)\otimes\pi^*L.
\]
Thus $(\alpha,\beta)$ is the correctly twisted quadratic matrix
factorization associated with $F$.

Locally on $X$, for a section $e\in E$, write
\[
\alpha(e):B_1\longrightarrow B_0
\]
and
\[
\beta(e):B_0\longrightarrow B_1\otimes L.
\]
The identities above imply
\[
\alpha(e)\beta(e)
=
f(e)\,\id_{B_0}.
\]
Polarizing this identity, and using $\operatorname{char}k\neq2$, gives
\[
\alpha(e)\beta(e')
+
\alpha(e')\beta(e)
=
b_f(e,e')\,\id_{B_0},
\]
where
\[
b_f(e,e')
=
f(e+e')-f(e)-f(e')
\]
is the polar form.

For a local section $\lambda\in L^\vee$, define
\[
\rho_F(e\otimes e'\otimes\lambda)
=
(\id_{B_0}\otimes\lambda)
\circ
(\alpha(e)\otimes\id_L)
\circ
\beta(e').
\]

We verify the tensorial relations defining the even Clifford
algebra. This verification is local on \(X\), so trivialize \(L\)
and put

\[
j_F(u,v)=\alpha(u)\beta(v).
\]
The two matrix-factorization identities give

\[
j_F(v,v)=f(v)\operatorname{id}_{B_0}.
\]
They also give

\[
\begin{aligned}
j_F(u,v)j_F(v,w)
&=
\alpha(u)\beta(v)\alpha(v)\beta(w)\\
&=
f(v)\alpha(u)\beta(w)\\
&=
f(v)j_F(u,w).
\end{aligned}
\]
Restoring the \(L\)- and \(L^\vee\)-factors gives the corresponding
tensorial relations for the line-bundle-valued even Clifford
algebra. Hence its universal property yields a unital algebra
homomorphism

\[
\rho_F:
\operatorname{Cl}_0(E,f,L)
\longrightarrow
\operatorname{End}_{\mathcal O_X}(B_0).
\]
This gives a functor
\[
\operatorname{Ulr}(Y/X)
\longrightarrow
\operatorname{Rep}_X(\mathscr C_0).
\]

We now construct the inverse locally on $X$.  Trivialize $L$ on an open
set $U\subset X$.  Since the quadratic form is nondegenerate, after
shrinking $U$ there is a section
\[
e_0\in E(U)
\]
such that
\[
f(e_0)\in\OO_X(U)^\times.
\]

Let $(B,\rho)$ be a locally free $\mathscr C_0$-module.  Put
\[
B_0=B_1=B.
\]
Using the trivialization of $L$, define
\[
\alpha(e)
=
\rho(e\otimes e_0)
\]
and
\[
\beta(e)
=
f(e_0)^{-1}\rho(e_0\otimes e).
\]
The Clifford relations give
\[
\alpha(e)\beta(e)=f(e)\,\id_B
\]
and, after polarization,
\[
\alpha\beta=f,
\qquad
\beta\alpha=f.
\]

The map

\[
\alpha:\pi^*B(-1)\longrightarrow\pi^*B
\]
is injective. Indeed, the matrix-factorization identity
\(\beta\alpha=fI\), with the appropriate twists, implies that
\(\alpha(s)=0\) forces \(fs=0\). Since \(f\) defines an effective
Cartier divisor and the source is locally free, multiplication
by \(f\) on the source is injective. Thus \(s=0\).

The same argument applies after passage to every geometric fibre:
the fibre equation is a non-zero-divisor on the projective-space
fibre, so the fibre map \(\alpha_{\bar x}\) is injective. Both
ambient vector bundles are flat over the base. The fibrewise
exactness criterion for flatness therefore shows that
\[
\operatorname{coker}(\alpha)=i_*F
\]
is flat over the base. Equivalently, \(F\) is flat over the base.

On each geometric fibre, the resulting sheaf is maximal
Cohen--Macaulay on the smooth quadric and hence locally free.
The fibrewise local-freeness criterion now implies that \(F\)
is locally free on \(Y\). Its linear presentation then shows,
by \cref{prop:relative-linear}, that it is relatively Ulrich.

We explain the local equivalence and its descent. Work on an
open set on which \(L\) is trivial and choose \(e_0\) with

\[
a=f(e_0)\in\mathcal O_X^\times.
\]
The orthogonal decomposition is

\[
E=\mathcal O_Xe_0\oplus V.
\]
Write

\[
f(we_0+v)=a\bigl(w^2-Q(v)\bigr),
\qquad
Q(v)=-a^{-1}f(v).
\]
The map

\[
V\longrightarrow \operatorname{Cl}_0(E,f,L),
\qquad
v\longmapsto a^{-1}e_0v,
\]
satisfies

\[
(a^{-1}e_0v)^2=Q(v).
\]
Its image generates the even Clifford algebra, and both the
resulting ordinary Clifford algebra and the even Clifford algebra
are locally free of rank \(2^n\). It therefore identifies the
local monic projection algebra with
\(\operatorname{Cl}_0(E,f,L)\).

For the inverse construction, put

\[
\mu(v)=a^{-1}\rho(e_0v).
\]
The first matrix-factorization map, divided by \(a\), is

\[
wI-M_\mu(y).
\]
Thus the inverse construction is the finite-projection
construction of \cref{thm:clifford-dictionary} under the preceding algebra
identification.

Conversely, for a relatively Ulrich bundle with matrix
factorization \((\alpha,\beta)\), multiplication by the monic
coordinate is

\[
-\alpha(v)\alpha(e_0)^{-1}
=
-a^{-1}\alpha(v)\beta(e_0).
\]
Since \(v\) is orthogonal to \(e_0\), polarization gives

\[
-\alpha(v)\beta(e_0)=\alpha(e_0)\beta(v).
\]
Hence this action equals

\[
a^{-1}\rho_F(e_0v).
\]
The intrinsic construction therefore agrees locally with the
equivalence of \cref{thm:clifford-dictionary}. In particular, it is locally fully
faithful and essentially surjective.

Now choose an open cover \(\{U_i\}\) on which these constructions
are available. For a global module \((B,\rho)\), let \(F_i\) be its
local inverse image on \(U_i\). On \(U_i\cap U_j\), full faithfulness
gives a unique isomorphism

\[
F_i|_{U_i\cap U_j}
\xrightarrow{\sim}
F_j|_{U_i\cap U_j}
\]
inducing the identity on the restricted Clifford module.
Uniqueness implies the cocycle condition on triple overlaps.
Vector-bundle descent therefore produces a global relatively
Ulrich bundle.

Exactness follows from the local finite-projection equivalence.
The construction is compatible with arbitrary base change:
the relative presentation, the Clifford algebra, and the
finite-projection equivalence commute with base change, and
the unique descent isomorphisms are compatible with them. Hence we obtain an equivalence
\[
\operatorname{Ulr}(Y/X)
\simeq
\operatorname{Rep}_X(\mathscr C_0).
\]
Finally, since $Y/X$ has degree $2$, \cref{prop:relative-linear}
gives
\[
\rk p_*F=2\rk F.
\]
Thus the coefficient bundle of the corresponding Clifford module is
exactly $p_*F$.
\end{proof}

\begin{theorem}[Existence of relatively Ulrich bundles on quadric fibrations]
\label{thm:quadric-existence}
Assume
\[
\operatorname{char}k\neq2.
\]
Let
\[
p:Y\longrightarrow X
\]
be a smooth quadric fibration of relative dimension
\[
n-1\geq1.
\]
Then $Y/X$ carries a relatively Ulrich bundle.

More precisely, there exists a relatively Ulrich bundle
\[
F_{\mathrm{reg}}
\]
of rank
\[
\rk F_{\mathrm{reg}}=2^{n-1}.
\]
Consequently,
\[
\uc(Y/X)\leq2^{n-1}.
\]
No projection centre is required.
\end{theorem}

\begin{proof}
By \cref{thm:quadric-clifford}, relatively Ulrich bundles on $Y/X$
are equivalent to locally free modules over the intrinsic even Clifford
algebra
\[
\mathscr C_0=\operatorname{Cl}_0(E,f,L).
\]

The even Clifford algebra of a nondegenerate quadratic form on a
rank-$(n+1)$ bundle is locally free of rank
\[
\rk_{\OO_X}\mathscr C_0=2^n.
\]
This follows from the standard PBW filtration for Clifford algebras;
see, for example, Bichsel--Knus or Knus.

Now $\mathscr C_0$ acts on itself by left multiplication:
\[
\mathscr C_0
\longrightarrow
\operatorname{End}_{\OO_X}(\mathscr C_0).
\]
Thus the regular module
\[
B_{\mathrm{reg}}:=\mathscr C_0
\]
is a locally free representation of $\mathscr C_0$ of rank
\[
2^n.
\]

By \cref{thm:quadric-clifford}, this representation corresponds to a
relatively Ulrich bundle $F_{\mathrm{reg}}$ on $Y$.  The rank formula
in that theorem gives
\[
\rk B_{\mathrm{reg}}
=
2\,\rk F_{\mathrm{reg}}.
\]
Therefore
\[
2^n
=
2\,\rk F_{\mathrm{reg}},
\]
and hence
\[
\rk F_{\mathrm{reg}}
=
2^{n-1}.
\]

Thus a relatively Ulrich bundle always exists and
\[
\uc(Y/X)\leq2^{n-1}.
\]
\end{proof}

\begin{proposition}[Comparison with the projection-centre algebra]
\label{prop:quadric-centre-comparison}
Assume that $2$ is invertible on $X$, and let
\[
f\in H^0\!\left(X,\Sym^2E^\vee\otimes L\right)
\]
define a smooth quadric fibration
\[
p:Y=V(f)\longrightarrow X.
\]
Assume in addition that
\[
\phi:S\hookrightarrow E
\]
is a relative projection centre.

Then $\phi$ is automatically split.  Let
\[
V:=S^\perp\subset E
\]
be the orthogonal complement with respect to the polar form of $f$.
Then
\[
E=S\oplus V.
\]
Via the isomorphism
\[
L\simeq S^{\otimes2}
\]
induced by the non-vanishing section $f(\phi)$, put
\[
W:=V\otimes S^{-1}.
\]
With respect to this orthogonal splitting, the equation of $Y$ has
the form
\[
w^2-Q=0
\]
for a nondegenerate quadratic form
\[
Q\in H^0(X,\Sym^2W^\vee).
\]

There is a canonical isomorphism of $\OO_X$-algebras
\[
C_{f,\phi}
\simeq
\Cl(W,Q)
\xrightarrow{\sim}
\Cl_0(E,f,L).
\]

More explicitly, for a local generator $s$ of $S$, the second
isomorphism is induced by the $\OO_X$-linear map
\[
\vartheta:W\longrightarrow\Cl_0(E,f,L)
\]
defined by
\[
\vartheta(v\otimes s^{-1})
=
s\otimes v\otimes
(s\otimes s)^{-1}.
\]

Under this isomorphism, the equivalence
\[
\Ulr(Y/X)
\xrightarrow{\sim}
\Rep_X(C_{f,\phi})
\]
of Theorem~\ref{thm:clifford-dictionary}
agrees with the centre-free equivalence
\[
\Ulr(Y/X)
\xrightarrow{\sim}
\Rep_X\!\left(\Cl_0(E,f,L)\right)
\]
of \cref{thm:quadric-clifford}.
\end{proposition}
\begin{proof}
We divide the proof into three steps.

\medskip
\noindent
\textbf{Step 1: the orthogonal splitting.}
Since $\phi:S\hookrightarrow E$ is a projection centre, the
restriction of $f$ to $S$ is nowhere zero.  Hence it induces an
isomorphism
\[
S^{\otimes2}\xrightarrow{\sim}L.
\]
Because $2$ is invertible, the restriction of the polar form
$b_f$ to $S$ is nondegenerate.  Therefore the orthogonal complement
\[
V:=S^\perp
\]
is a vector subbundle and
\[
E=S\oplus V.
\]

Using the isomorphism $L\simeq S^{\otimes2}$, and putting
\[
W=V\otimes S^{-1},
\]
the quadratic form takes the form
\[
f=w^2-Q
\]
for a uniquely determined quadratic form
\[
Q\in H^0(X,\Sym^2W^\vee).
\]
Since $Y\to X$ is smooth, $Q$ is nondegenerate.

By the definition of the projection-centre algebra in degree two,
the monic relation is
\[
M(y)^2=Q(y)I.
\]
After polarization this is precisely the ordinary Clifford relation.
Thus
\[
C_{f,\phi}\simeq\Cl(W,Q).
\]

\medskip
\noindent
\textbf{Step 2: comparison with the even Clifford algebra.}
We construct an isomorphism
\[
\Cl(W,Q)\longrightarrow\Cl_0(E,f,L).
\]

Let $s$ be a local generator of $S$.  For a local section
$v\in V$, define
\[
\vartheta(v\otimes s^{-1})
=
s\otimes v\otimes(s\otimes s)^{-1}.
\]
The right-hand side is a local section of the even Clifford algebra.

This definition is independent of the choice of $s$.  Indeed, if
\[
s'=us
\]
for a unit $u$, then the same section of $W$ is written as
\[
v\otimes s^{-1}
=
(uv)\otimes(s')^{-1},
\]
and
\[
s'\otimes(uv)\otimes(s'\otimes s')^{-1}
=
s\otimes v\otimes(s\otimes s)^{-1}.
\]
Hence $\vartheta$ is globally well defined.

It remains to verify the Clifford relation.  This may be checked
locally, so trivialize $S$ and $L$ by a section $s$ with
\[
f(s)=1.
\]
Then
\[
f(v)=-Q(v)
\qquad (v\in V),
\]
and $s$ is orthogonal to $V$.  In the ordinary local Clifford
algebra we therefore have
\[
sv=-vs,
\qquad
s^2=1,
\qquad
v^2=-Q(v).
\]
Consequently
\[
(sv)^2
=
svsv
=
-s^2v^2
=
Q(v).
\]
Thus
\[
\vartheta(\xi)^2=Q(\xi)
\qquad
(\xi\in W),
\]
so the universal property of the Clifford algebra gives a homomorphism
\[
\Cl(W,Q)\longrightarrow\Cl_0(E,f,L).
\]

This map is surjective.  Indeed, locally choose a basis
\[
v_1,\ldots,v_n
\]
of $V$.  The even Clifford algebra is generated by products of two
vectors.  The elements
\[
sv_i=\vartheta(v_i)
\]
belong to the image, and
\[
v_iv_j
=
-\vartheta(v_i)\vartheta(v_j)
\]
also belongs to the image.  Hence all even products are generated by
the image of $\vartheta$.

Both algebras are locally free of rank
\[
2^n.
\]
For $\Cl(W,Q)$ this follows from the PBW theorem, since
$\rk W=n$.  For $\Cl_0(E,f,L)$ it follows from the corresponding
PBW theorem for the even Clifford algebra of the rank-$(n+1)$
quadratic bundle $E$.  Therefore the surjective homomorphism is an
isomorphism:
\[
\Cl(W,Q)\xrightarrow{\sim}\Cl_0(E,f,L).
\]

\medskip
\noindent
\textbf{Step 3: compatibility of the two Ulrich--Clifford
correspondences.}
Let $F$ be a relatively Ulrich bundle.  Write its relative linear
presentation as
\[
0\longrightarrow
\pi^*B_1(-1)
\xrightarrow{\alpha}
\pi^*B_0
\longrightarrow
i_*F
\longrightarrow0,
\qquad
B_0\simeq p_*F.
\]
Let $\beta$ be the second map in the associated quadratic matrix
factorization.

Choose locally a generator $s$ of $S$ with $f(s)=1$, and a local
basis
\[
v_1,\ldots,v_n
\]
of $V$.  With corresponding fibre coordinates $w,y_1,\ldots,y_n$,
the first matrix is
\[
\alpha
=
w\,\alpha(s)
+
\sum_i y_i\alpha(v_i).
\]

Since
\[
\alpha(s)\beta(s)=\id_{B_0},
\qquad
\beta(s)\alpha(s)=\id_{B_1},
\]
the map $\alpha(s)$ is an isomorphism and
\[
\alpha(s)^{-1}=\beta(s).
\]

In $\coker(\alpha)=i_*F$ one has, for $b_1\in B_1$,
\[
w\,\alpha(s)b_1
=
-\sum_i y_i\alpha(v_i)b_1.
\]
Writing
\[
b_0=\alpha(s)b_1
\]
and hence
\[
b_1=\beta(s)b_0,
\]
we obtain
\[
w\,b_0
=
-\sum_i y_i\alpha(v_i)\beta(s)b_0.
\]
Therefore the action defining the projection-centre representation is
\[
\mu_F(v_i\otimes s^{-1})
=
-\alpha(v_i)\beta(s).
\]

On the other hand, polarizing the matrix-factorization identity gives
\[
\alpha(v_i)\beta(s)
+
\alpha(s)\beta(v_i)
=
b_f(v_i,s)\,\id.
\]
Since $V=S^\perp$,
\[
b_f(v_i,s)=0,
\]
and therefore
\[
-\alpha(v_i)\beta(s)
=
\alpha(s)\beta(v_i).
\]

By the definition of the centre-free even-Clifford action,
\[
\rho_F\!\left(
\vartheta(v_i\otimes s^{-1})
\right)
=
\alpha(s)\beta(v_i).
\]
Hence
\[
\mu_F(v_i\otimes s^{-1})
=
\rho_F\!\left(
\vartheta(v_i\otimes s^{-1})
\right).
\]

Thus the two representations correspond under
\[
C_{f,\phi}
\simeq
\Cl_0(E,f,L),
\]
and consequently the two equivalences of categories agree.
\end{proof}

\begin{remark}
The point of \cref{thm:quadric-existence} is that, unlike the projection-centre construction of \cref{thm:clifford-dictionary}, it does not require a global projection centre. For quadrics, the intrinsic even Clifford algebra
\[
\operatorname{Cl}_0(E,f,L)
\]
provides a globally defined algebra governing relatively Ulrich
bundles.
\end{remark}

\subsection{A basic source of examples}

The next observation gives another simple class of relative families
for which existence follows immediately from the absolute case.

\begin{proposition}[Constant families]
\label{prop:constant-family-existence}
Let
\[
Z\subset\PP^n_k
\]
be a projective hypersurface equipped with an Ulrich bundle $G$, and
let $X$ be a $k$-scheme.  Consider the constant family
\[
p:Y=X\times Z\longrightarrow X
\]
with relative polarization
\[
\OO_Y(1)=\operatorname{pr}_2^*\OO_Z(1).
\]
Then
\[
F:=\operatorname{pr}_2^*G
\]
is relatively Ulrich on $Y/X$.
\end{proposition}

\begin{proof}
Let $m=\dim Z$.  For $1\leq j\leq m$,
\[
F(-j)
=
\operatorname{pr}_2^*(G(-j)).
\]
Since $p$ is the projection onto the first factor,
\[
Rp_*F(-j)
\simeq
R\Gamma(Z,G(-j))\otimes_k\OO_X.
\]
Because $G$ is Ulrich,
\[
R\Gamma(Z,G(-j))=0
\qquad
(1\leq j\leq m).
\]
Consequently
\[
Rp_*F(-j)=0
\qquad
(1\leq j\leq m),
\]
so $F$ is relatively Ulrich.
\end{proof}

\begin{remark}[Fibrewise existence versus relative existence]
\label{rem:fibrewise-relative-existence}
The preceding results also explain an important difference between
absolute and relative Ulrich theory.

If $F$ is relatively Ulrich, then
\cref{prop:fibrewise-ulrich} implies that every geometric fibre
$F_{\bar x}$ is Ulrich.  The converse statement with a prescribed rank
is not automatic: choosing an Ulrich bundle separately on every fibre
does not produce descent data for a vector bundle on $Y$.

Thus one should distinguish the following questions:
\[
\begin{array}{c}
\text{existence of Ulrich bundles on all fibres},\\[3pt]
\text{existence of a relatively Ulrich bundle on }Y/X,\\[3pt]
\text{existence of a relatively Ulrich bundle of minimal fibrewise
rank}.
\end{array}
\]
For smooth quadric fibrations in characteristic different from
\(2\), \cref{thm:quadric-existence} gives a positive answer to the second question
without requiring a projection center. The later analysis of the discriminant cover and the Brauer class
shows that the third question is subtler: monodromy and Brauer
obstructions may force the minimal relative rank to be strictly larger
than the minimal rank on the geometric fibers.
\end{remark}

\section{Stability and irreducibility}

\paragraph{Purpose of the section.}
We show that the usual stability of an Ulrich bundle is exactly irreducibility of the associated Clifford representation.  This converts geometric Jordan--Holder theory into ordinary representation theory.

Let $Z\subset\PP^n_K$ be a smooth hypersurface with $n\ge2$ and $K$ algebraically closed.

\begin{proposition}[Casanellas--Hartshorne]\label{prop:ulrich-semistable}
Every Ulrich bundle on $Z$ is Gieseker semistable.  Moreover, if
\[
0\longrightarrow G'\longrightarrow G\longrightarrow G''\longrightarrow0
\]
is exact, $G$ is Ulrich, $G''$ is torsion free, and $p_{G'}=p_G$, then $G'$ and $G''$ are Ulrich.  In particular every Jordan--Holder factor of an Ulrich bundle is Ulrich.
\end{proposition}

\begin{proof}
These statements are due to Casanellas--Hartshorne \cite[Thm.~2.9]{CH12}.  
\end{proof}

\begin{definition}
A relatively Ulrich bundle $F$ is \emph{relatively stable} if every geometric fibre is Gieseker stable.  It is \emph{relatively simple} if
\[
\OO_X\xrightarrow{\sim}p_*\End(F).
\]
A representation $(B,\rho)$ is \emph{fibrewise irreducible} if every geometric fibre is irreducible.
\end{definition}

\begin{theorem}[Stability dictionary]\label{thm:stability-dictionary}
Assume the hypotheses of \cref{thm:clifford-dictionary} and $n\ge2$.  Let
\[
(B,\rho)=\Lambda_\phi(F).
\]
Then:
\begin{enumerate}
\item for every geometric point $\bar x\to X$, $F_{\bar x}$ is Gieseker stable if and only if $\rho_{\bar x}$ is irreducible;
\item $F$ is relatively stable if and only if $\rho$ is fibrewise irreducible;
\item if $X$ is connected and $\rho$ is fibrewise irreducible, then $(B,\rho)$ has no nonzero proper admissible subrepresentation;
\item if $X$ is integral and the geometric generic representation is irreducible, then
\[
\End_{C_{f,\phi}}(B,\rho)\cong p_*\End(F)\cong\OO_X.
\]
If $X$ is also proper, then $\End_{Y_f}(F)=k$.
\end{enumerate}
\end{theorem}

\begin{proof}
After base change to a geometric point, \cref{thm:clifford-dictionary} is an equivalence between Ulrich bundles and finite-dimensional representations.  If the representation has a nontrivial invariant subspace, its inverse image is a nontrivial Ulrich subbundle with the same reduced Hilbert polynomial, so the Ulrich bundle is not stable.  Conversely, if the Ulrich bundle is not stable, semistability from \cref{prop:ulrich-semistable} gives a saturated subsheaf with the same reduced Hilbert polynomial; the proposition shows that the subobject and quotient are Ulrich, hence the Clifford representation has a proper invariant subspace.  This proves (1), and (2) is immediate.

For (3), an admissible subrepresentation is a subbundle.  Its fibre rank is locally constant, and on every geometric fibre irreducibility forces that rank to be either $0$ or $\rk B$.  Connectedness makes the choice global.

For (4), the sheaf of commuting endomorphisms is the kernel of
\[
\End(B)\longrightarrow\Hom(W\otimes B,B),
\qquad u\mapsto\bigl[v\otimes b\mapsto u\rho(v)b-\rho(v)ub\bigr].
\]
By the equivalence, this kernel is $p_*\End(F)$.  It contains the scalar copy of $\OO_X$.  The quotient by the scalars is torsion free because it is a subsheaf of the vector bundle $\End(B)/\OO_X$.  At the geometric generic point, Schur's lemma says that the commuting endomorphisms are only scalars.  Hence the quotient is both torsion and torsion free, so it is zero.  If $X$ is proper integral over the algebraically closed field $k$, then $H^0(X,\OO_X)=k$.
\end{proof}

\begin{remark}
Global irreducibility in the exact category is weaker than fibrewise irreducibility.  A fibrewise invariant summand can fail to descend because monodromy permutes the summands.  The quadric-surface family in \cref{prop:monodromy-example} gives an explicit example.
\end{remark}

\section{The relative moduli stack over the base}
\label{sec:relative-moduli}

We study the moduli problem over the base $X$.
For a morphism
\[
T\longrightarrow X,
\]
the objects parametrized below are relatively Ulrich bundles on the
base-changed family
\[
Y_f\times_X T\longrightarrow T.
\]
The Clifford description gives this relative moduli stack as an
explicit quotient stack over $X$. In characteristic zero, the relative quotient stack has an affine
good moduli space over $X$. In the absolute case
$X=\Spec k$, this description also gives the Ulrich locus inside the
usual Gieseker moduli space.

\begin{definition}
For $N\ge1$, let $R_N=R_N(f,\phi)$ be the functor whose $T$-points, for $T\to X$, are maps
\[
\mu:W_T\longrightarrow M_N(\OO_T)
\]
satisfying the monic identity.  The group $\mathrm{GL}_N$ acts by conjugation.
\end{definition}

\begin{lemma}\label{lem:representation-scheme}
The functor $R_N$ is represented by a closed subscheme of the total space of the vector bundle $\Hom(W,M_N)$; in particular it is affine and of finite type over $X$.  If $d\nmid N$, then $R_N$ is empty.
\end{lemma}

\begin{proof}
The total space of $\Hom(W,M_N)$ represents all $\OO_T$-linear maps $W_T\to M_N(\OO_T)$.  The coefficients of the matrix polynomial $P_\mu(y)$ are finitely many polynomial equations, so their common zero locus is closed.  If a field-valued point existed with $d\nmid N$, \cref{thm:clifford-dictionary} over that field would produce an Ulrich bundle of rank $N/d$, impossible.
\end{proof}

Let $\Ulr_r(Y_f/X)$ be the fppf stack over $X$ defined as follows.
For a morphism
\[
T\longrightarrow X,
\]
its objects are relatively Ulrich bundles of rank $r$ on
\[
Y_f\times_XT\longrightarrow T.
\]
Thus its fibre over a geometric point $\bar x\to X$ parametrizes
Ulrich bundles on the single hypersurface $Y_{f,\bar x}$.

This should not be confused with the different moduli problem over
$k$ whose objects are global vector bundles on the whole family
$Y_f$ which are relatively Ulrich over $X$.

\begin{theorem}[Quotient-stack presentation]\label{thm:ulrich-stack}
There is a natural isomorphism of stacks over $X$
\[
\mathscr U\!lr_r(Y_f/X)\cong [R_{dr}/\mathrm{GL}_{dr}].
\]
Consequently the stack is algebraic, of finite type over $X$, with affine diagonal.
\end{theorem}
\begin{proof}
Let
\[
T\longrightarrow X
\]
be a test scheme and put
\[
N=dr.
\]
A $T$-point of the quotient stack
\[
[R_N/\GL_N]
\]
is a $\GL_N$-torsor on $T$ together with a $\GL_N$-equivariant map
to $R_N$.

Equivalently, it is a rank-$N$ vector bundle $B$ on $T$ together
with an algebra homomorphism
\[
C_{f_T,\phi_T}
\longrightarrow
\End_{\OO_T}(B),
\]
that is, a locally free representation of the base-changed monic
Clifford algebra.

By the base-change-compatible Clifford equivalence of
Theorem~4.6, such a representation is equivalent to a relatively
Ulrich bundle of rank
\[
\frac{N}{d}=r
\]
on
\[
Y_f\times_XT\longrightarrow T.
\]

This correspondence is functorial in $T$, and it identifies
isomorphisms and fppf descent data. Hence
\[
\Ulr_r(Y_f/X)
\simeq
[R_{dr}/\GL_{dr}]
\]
as stacks over $X$.

Since $R_{dr}$ is affine and of finite type over $X$, the quotient
stack is algebraic and of finite type over $X$, with affine
diagonal.
\end{proof}

\begin{theorem}[Good moduli space in characteristic zero]\label{thm:good-moduli}
Assume $\charac k=0$ and put $N=dr$. Let $a:R_N\longrightarrow X$ then:
\begin{enumerate}

\item
$\Ulr_r(Y_f/X)$ has a good moduli space
\[
M_r(Y_f/X)
=
R_N/\!/\GL_N
:=
\Spec_X\!\left(
(a_*\OO_{R_N})^{\GL_N}
\right),
\]
which is affine over $X$;
\item over a geometric point $\bar x$, the points of $M_r$ correspond to isomorphism classes of polystable Ulrich bundles of rank $r$ on $Y_{f,\bar x}$;
\item the locus of fibrewise absolutely irreducible representations is open.  It corresponds to relatively stable Ulrich bundles, and its quotient is quasi-affine over $X$; the representation space maps to that quotient as a $\mathrm{PGL}_N$-torsor in the fppf topology.
\end{enumerate}
\end{theorem}

\begin{proof}
Since $\operatorname{char}k=0$, the group $\GL_N$ is linearly
reductive. Because
\[
a:R_N\longrightarrow X
\]
is affine, invariant theory gives the relative affine quotient
\[
R_N/\!/\GL_N
=
\Spec_X\!\left(
(a_*\OO_{R_N})^{\GL_N}
\right).
\]
This is a good moduli space for
\[
[R_N/\GL_N];
\]
see \cite{Alp13}. Together with \cref{thm:ulrich-stack}, this proves (1).

Over an algebraically closed field, closed orbits in the representation scheme are precisely semisimple representations; see \cite{Art69,Pro74}.  Under \cref{thm:stability-dictionary}, irreducible representations correspond to stable Ulrich bundles, so semisimple representations correspond to direct sums of stable Ulrich bundles.  This proves (2).

Absolute irreducibility is open: locally, Burnside's theorem says that a representation is absolutely irreducible exactly when the algebra generated by the matrices $\mu(W)$ is all of $M_N$.  Words of bounded length already span $M_N$, and surjectivity of the resulting map of coherent sheaves is an open condition.  The standard Artin--Procesi theory of irreducible representations gives the $\mathrm{PGL}_N$-torsor statement and the geometric quotient.  The stable locus is therefore an open subspace of the affine good moduli space, hence quasi-affine.
\end{proof}

\begin{theorem}[Affineness of the absolute Ulrich locus]\label{thm:absolute-affine}
Assume $\charac k=0$.  Let $Z\subset\PP^n_k$ be a smooth hypersurface of degree $d$, with $n\ge2$, and let
\[
P_r(t)=dr\binom{t+n-1}{n-1}.
\]
Inside the Gieseker moduli space $M_Z(P_r)$ of semistable sheaves with Hilbert polynomial $P_r$, the locus $U_r(Z)$ of S-equivalence classes represented by Ulrich bundles is open and affine.  Its stable locus is quasi-affine.
\end{theorem}

\begin{proof}
Choose a point of $\PP^n\setminus Z$ as a split projection centre.  The stack of semistable sheaves with Hilbert polynomial $P_r$ has the usual GIT good moduli space $M_Z(P_r)$; see \cite{HL10,Alp13}.  The substack of Ulrich bundles is open by cohomological semicontinuity and openness of local freeness.

It is also saturated for S-equivalence.  Indeed, if $G$ is Ulrich, all its Jordan--Holder factors are Ulrich by \cref{prop:ulrich-semistable}.  Any semistable sheaf S-equivalent to $G$ is an iterated extension of the same stable factors.  Extensions of vector bundles are vector bundles, and the defining Ulrich vanishings are preserved by extensions; hence every sheaf in the S-equivalence class is again Ulrich.  Therefore the open Ulrich substack descends to an open subscheme $U_r(Z)$ of the Gieseker moduli space.

On the other hand, \cref{thm:ulrich-stack,thm:good-moduli} identify the same stack with $[R_{dr}/\mathrm{GL}_{dr}]$ and its good moduli space with the affine scheme $R_{dr}//\mathrm{GL}_{dr}$. By uniqueness of good moduli spaces
\cite[Thm.~6.6]{Alp13},
the two good moduli spaces are canonically isomorphic. Hence
\[
U_r(Z)
\simeq
R_{dr}/\!/\GL_{dr}.
\]
The stable locus is open in this affine scheme.
\end{proof}

\begin{example}[Plane curves]
Let $Z\subset\PP^2$ be a smooth plane curve of genus $g$.  A line bundle $L$ is Ulrich precisely when
\[
H^*(Z,L(-1))=0.
\]
Then $\deg L(-1)=g-1$ and $L(-1)$ lies outside the theta divisor.  Thus \cref{thm:absolute-affine} recovers the classical affine variety
\[
\Pic^{g-1}(Z)\setminus\Theta.
\]
\end{example}

\begin{remark}
In positive characteristic, $\mathrm{GL}_N$ is geometrically reductive rather than linearly reductive.  Adequate moduli spaces are the natural replacement.  We do not pursue that extension here.
\end{remark}

\section{Ulrich complexity}

\paragraph{Purpose of the section.}
We introduce three rank invariants and record the formal inequalities that distinguish fibrewise, relative, and normalized existence.

\begin{definition}
For a smooth relative hypersurface $Y_f/X$ set
\[
\uc(Y_f/X)=\min\{\rk F:0\ne F\in\Ulr(Y_f/X)\},
\]
\[
\ucnor(Y_f/X)=\min\{\rk F:0\ne F\in\Ulr(Y_f/X)\text{ normalized}\},
\]
with minimum of the empty set equal to $\infty$. For a split centre $\phi$, put
\[
\cc(f,\phi)
=
\min\left\{
\rk B:
0\ne(B,\rho)\in\Rep_X(C_{f,\phi})
\right\}.
\]

We also define the \emph{normalized Clifford complexity} by
\[
\ccnor(f,\phi)
=
\min\left\{
N>0:
\text{there exists an } \OO_X\text{-algebra homomorphism }
\rho:C_{f,\phi}\longrightarrow
\End_{\OO_X}(\OO_X^{\oplus N})
\right\}.
\]
As usual, the minimum of the empty set is understood to be $\infty$.
\end{definition}

\begin{proposition}
\label{prop:complexities}
Let $Y_f\to X$ be smooth.

\begin{enumerate}
\item
\[
\uc(Y_f/X)\leq \ucnor(Y_f/X),
\]
and the set of relative Ulrich ranks is closed under addition.

\item If $\phi$ is a split centre, then
\[
\cc(f,\phi)=d\,\uc(Y_f/X).
\]

\item If $\phi$ is a split centre, then
\[
\ccnor(f,\phi)=d\,\ucnor(Y_f/X).
\]

\item For every geometric point $\bar x\to X$,
\[
\uc(Y_{f,\bar x})\leq \uc(Y_f/X).
\]

\item Relative and normalized complexity can only decrease after
base change.

\item For a constant family
\[
X\times Z\longrightarrow X,
\]
one has
\[
\uc(X\times Z/X)=\uc(Z).
\]

\item If the fibres have dimension at least $3$ and $d\geq2$, then
every fibre has Ulrich complexity at least $2$.  In particular,
there is no relatively Ulrich line bundle on $Y_f/X$.
\end{enumerate}
\end{proposition}

\begin{proof}
For (1), every normalized relatively Ulrich bundle is, in particular,
a relatively Ulrich bundle.  Therefore
\[
\uc(Y_f/X)\leq\ucnor(Y_f/X).
\]
Here, as in the definition, the minimum of the empty set is understood
to be $\infty$.

Now suppose that $F$ and $G$ are relatively Ulrich bundles of ranks
$r$ and $s$.  Since
\[
R p_*(F(-j))=0
\qquad\text{and}\qquad
R p_*(G(-j))=0
\]
for
\[
1\leq j\leq n-1,
\]
we also have
\[
R p_*\bigl((F\oplus G)(-j)\bigr)=0.
\]
Thus $F\oplus G$ is relatively Ulrich of rank $r+s$.  Hence the set
of relative Ulrich ranks is closed under addition.

For (2), assume that $\phi$ is a split centre.  By
\cref{thm:clifford-dictionary}, a relatively Ulrich bundle $F$ of
rank $r$ corresponds to a representation
\[
(B,\rho)\in\Rep_X(C_{f,\phi})
\]
with
\[
B\simeq p_*F
\]
and
\[
\rk B=d\,r.
\]
Conversely, every locally free representation of $C_{f,\phi}$ arises
in this way.  Therefore the possible ranks of nonzero representations
are exactly $d$ times the possible relative Ulrich ranks.  Taking the
minimum gives
\[
\cc(f,\phi)=d\,\uc(Y_f/X).
\]
The same argument also covers the case in which both sides are
$\infty$.

For (3), we restrict the preceding correspondence to normalized
objects.  By the normalized Clifford correspondence,
\[
F\text{ is normalized}
\quad\Longleftrightarrow\quad
p_*F\simeq\OO_X^{\oplus d\rk F},
\]
and this is equivalent to the corresponding representation having
trivial coefficient bundle
\[
B\simeq\OO_X^{\oplus d\rk F}.
\]
Thus a normalized relatively Ulrich bundle of rank $r$ corresponds
exactly to a representation of $C_{f,\phi}$ on
\[
\OO_X^{\oplus dr}.
\]
Taking minimum ranks gives
\[
\ccnor(f,\phi)
=
d\,\ucnor(Y_f/X).
\]

For (4), let $F$ be a relatively Ulrich bundle of rank $r$ on
$Y_f/X$.  By the fibrewise criterion, for every geometric point
\[
\bar x\to X,
\]
the restriction
\[
F_{\bar x}
\]
is an Ulrich bundle of rank $r$ on $Y_{f,\bar x}$.  Hence
\[
\uc(Y_{f,\bar x})\leq r.
\]
Taking the minimum over all relatively Ulrich bundles $F$ gives
\[
\uc(Y_{f,\bar x})
\leq
\uc(Y_f/X).
\]

For (5), let
\[
g:X'\longrightarrow X
\]
be a nonempty base change and put
\[
Y_f'=Y_f\times_X X'.
\]
If $F$ is relatively Ulrich on $Y_f/X$, then its pullback
\[
F'=F|_{Y_f'}
\]
is relatively Ulrich on $Y_f'/X'$ by base change.  Its rank is the
same as that of $F$.  Therefore
\[
\uc(Y_f'/X')
\leq
\uc(Y_f/X).
\]

If $F$ is normalized, then
\[
p_*F\simeq\OO_X^{\oplus d\rk F}.
\]
After base change,
\[
p'_*F'
\simeq
g^*(p_*F)
\simeq
\OO_{X'}^{\oplus d\rk F}.
\]
Thus $F'$ is again normalized, and hence
\[
\ucnor(Y_f'/X')
\leq
\ucnor(Y_f/X).
\]

For (6), consider the constant family
\[
p:X\times Z\longrightarrow X.
\]
If $G$ is an Ulrich bundle on $Z$ of rank $r$, then
\[
\pr_2^*G
\]
is relatively Ulrich. Hence
\[
\uc(X\times Z/X)\leq\uc(Z).
\]

Conversely, let
\[
\bar x=\Spec\Omega\longrightarrow X
\]
be a geometric point. Then
\[
(X\times Z)_{\bar x}\simeq Z_\Omega.
\]
Thus every relatively Ulrich bundle of rank $r$ restricts to an
Ulrich bundle of rank $r$ on $Z_\Omega$, so
\[
\uc(Z_\Omega)\leq\uc(X\times Z/X).
\]

We claim that
\[
\uc(Z_\Omega)=\uc(Z).
\]
Pullback gives
\[
\uc(Z_\Omega)\leq\uc(Z).
\]
Conversely, let $H_\Omega$ be an Ulrich bundle of rank $r$ on
$Z_\Omega$. Since $H_\Omega$ is of finite presentation, it descends
to
\[
Z_A:=Z\times_k\Spec A
\]
for some finitely generated $k$-subalgebra $A\subset\Omega$. After
shrinking $\Spec A$, local freeness and the finitely many Ulrich
cohomology vanishings remain valid. Since $k$ is algebraically
closed, this nonempty open set has a closed $k$-point. Specializing
there gives an Ulrich bundle of rank $r$ on $Z$. Hence
\[
\uc(Z)\leq\uc(Z_\Omega).
\]
Therefore
\[
\uc(Z_\Omega)=\uc(Z),
\]
and so
\[
\uc(X\times Z/X)=\uc(Z).
\]

If a relatively
Ulrich bundle existed on $X\times Z/X$, its restriction to a
geometric fibre would give an Ulrich bundle on $Z$.

Finally, we prove (7).  Let
\[
Z=Y_{f,\bar x}\subset\mathbf P^n
\]
be a smooth geometric fibre of dimension at least $3$ and degree
$d\geq2$.

By the Grothendieck--Lefschetz theorem,
\[
\Pic(Z)\simeq\mathbb Z\cdot\OO_Z(1);
\]
see \cite{SGA2}.  Hence every line bundle on $Z$ is of the form
\[
\OO_Z(a)
\]
for some $a\in\mathbb Z$.

Suppose that $\OO_Z(a)$ were Ulrich.  The Ulrich vanishing for the
first twist gives
\[
H^0\bigl(Z,\OO_Z(a-1)\bigr)=0.
\]
If $a\geq1$, then
\[
H^0\bigl(Z,\OO_Z(a-1)\bigr)\neq0,
\]
because $a-1\geq0$.  Hence
\[
a\leq0.
\]

It follows that
\[
h^0\bigl(Z,\OO_Z(a)\bigr)\leq1.
\]
On the other hand, an Ulrich bundle of rank $r$ on a degree-$d$
hypersurface has
\[
h^0=dr.
\]
Indeed, this follows from the linear presentation, or equivalently
from \cref{prop:relative-linear} applied over a point.
For an Ulrich line bundle, $r=1$, so
\[
h^0=d\geq2.
\]
This contradicts
\[
h^0\bigl(Z,\OO_Z(a)\bigr)\leq1.
\]

Therefore a smooth hypersurface of dimension at least $3$ and degree
$d\geq2$ has no Ulrich line bundle.  Hence
\[
\uc(Y_{f,\bar x})\geq2
\]
for every geometric fibre.  By (4), there can be no relatively Ulrich
line bundle on $Y_f/X$.
\end{proof}

\begin{remark}
\label{rem:two-clifford-complexities}
The two Clifford complexities measure different representation
problems.

The invariant
\[
\cc(f,\phi)
\]
allows representations on arbitrary vector bundles on $X$, whereas
\[
\ccnor(f,\phi)
\]
requires the coefficient bundle to be trivial.  Therefore
\[
\cc(f,\phi)\leq\ccnor(f,\phi),
\]
and the inequality may be strict.

By \cref{prop:complexities},
\[
\frac{\cc(f,\phi)}{d}
=
\uc(Y_f/X),
\qquad
\frac{\ccnor(f,\phi)}{d}
=
\ucnor(Y_f/X).
\]
Thus the difference between the two Clifford complexities is exactly
the difference between relatively Ulrich bundles and normalized
relatively Ulrich bundles.

In particular,
\[
\ucnor(Y_f/X)=\infty
\]
if and only if $C_{f,\phi}$ has no positive-rank representation on a
trivial coefficient bundle.
\end{remark}

\begin{remark}
\label{rem:relative-vs-fibre-complexity}
The inequality
\[
\uc(Y_{f,\bar x})
\leq
\uc(Y_f/X)
\]
can be strict.

For quadric fibrations, Section~11 identifies two different mechanisms
which can produce this strict inequality: monodromy of the simple
Clifford modules and a nontrivial Brauer class.  Section~12 gives
explicit examples of both phenomena.
\end{remark}

\section{Strict wildness from the base}

\paragraph{Purpose of the section.}
We prove a strong relative wildness result.  The argument is deliberately explicit: first embed modules over the free algebra into a Kronecker category, then realize the Kronecker category as extensions of two bundles on the base, and finally tensor with a relatively simple Ulrich bundle.

Assume in this section that $X$ is proper over $k$.

\begin{definition}
An exact $k$-linear category $\mathcal C$ is \emph{strictly wild} if there is an exact fully faithful $k$-linear functor
\[
\operatorname{mod}k\langle x,y\rangle\longrightarrow\mathcal C.
\]
We call $Y_f/X$ \emph{Ulrich wild} if it carries families of pairwise non-isomorphic indecomposable relatively Ulrich bundles of arbitrarily large parameter dimension.
\end{definition}

\begin{lemma}\label{lem:all-algebras-free-two}
Let $\Lambda$ be a finitely generated $k$-algebra.  There is an exact fully faithful functor
\[
\operatorname{mod}\Lambda\longrightarrow\operatorname{mod}k\langle x,y\rangle.
\]
\end{lemma}

\begin{proof}
Choose generators $a_1,\dots,a_s$ of $\Lambda$ and distinct scalars $c_1,\dots,c_s\in k$.  For a finite-dimensional $\Lambda$-module $M$, take $M^{\oplus s}$.  Let $x$ act diagonally by the scalars $c_i$.  Let $y$ have diagonal blocks $a_i$, identity blocks in positions $(i,i+1)$, and zero elsewhere.  A map commuting with $x$ is block diagonal because the eigenvalues are distinct.  Commuting with the identity superdiagonal forces all diagonal blocks to be equal, and commuting with the diagonal blocks $a_i$ says that the common block is $\Lambda$-linear.  Thus the construction is full and faithful, and exactness is immediate.
\end{proof}

Let $K_m$ be the Kronecker quiver with $m$ arrows from a vertex $V$ to a vertex $U$.

\begin{lemma}\label{lem:free-to-kronecker}
For $m\ge3$, the functor
\[
(M;x,y)\longmapsto (M,M;\id_M,x,y,0,\dots,0)
\]
from $\operatorname{mod}k\langle x,y\rangle$ to $\operatorname{rep}(K_m)$ is exact and fully faithful.
\end{lemma}

\begin{proof}
A morphism of the displayed Kronecker representations is a pair $(u,v)$ satisfying $u=v$ from the identity arrow, and then $ux=x'u$ and $uy=y'u$.  Thus it is exactly a module homomorphism.
\end{proof}

\begin{proposition}[Extension embedding]\label{prop:extension-embedding}
Let $A,B$ be vector bundles on $X$ such that
\[
\End(A)=\End(B)=k,
\qquad \Hom(A,B)=\Hom(B,A)=0,
\]
and put $m=\dim_k\Ext^1(B,A)$.  If $m\ge3$, there is an exact fully faithful functor
\[
\operatorname{rep}(K_m)\longrightarrow\VB(X).
\]
\end{proposition}

\begin{proof}
Choose a basis $\varepsilon_1,\dots,\varepsilon_m$ of $\Ext^1(B,A)$.  For a Kronecker representation
\[
\theta=(U,V;\theta_1,\dots,\theta_m)
\]
put
\[
e_\theta=\sum_i\theta_i\otimes\varepsilon_i
\in \Hom_k(V,U)\otimes\Ext^1(B,A)
\cong\Ext^1(B\otimes V,A\otimes U).
\]
Choose an extension
\[
0\to A\otimes U\to G_\theta\to B\otimes V\to0
\]
with class $e_\theta$.

Suppose $\varphi:G_\theta\to G_{\theta'}$ is a morphism.  The composite
\[
A\otimes U\to G_\theta\xrightarrow{\varphi}G_{\theta'}\to B\otimes V'
\]
vanishes because $\Hom(A,B)=0$.  Hence $\varphi$ induces maps on the subbundle and quotient.  Simplicity of $A$ and $B$ makes them $\id_A\otimes u$ and $\id_B\otimes v$.  Naturality of extension classes gives
\[
(\id_A\otimes u)_*e_\theta=(\id_B\otimes v)^*e_{\theta'}
\]
if and only if $u\theta_i=\theta_i'v$ for all $i$.  Thus every morphism of extensions determines a unique Kronecker morphism.

Conversely, a Kronecker morphism gives equality of extension classes and hence a lift $G_\theta\to G_{\theta'}$.  Two lifts differ by a map $B\otimes V\to A\otimes U'$, which is zero because $\Hom(B,A)=0$.  Therefore the lift is unique.  This proves full faithfulness.  Exactness follows from the nine lemma applied to a short exact sequence of Kronecker representations and the corresponding diagram of extensions.
\end{proof}

Fix now a relatively simple $F\in\Ulr(Y_f/X)$ and define
\[
T_F:\VB(X)\longrightarrow\Ulr(Y_f/X),
\qquad G\longmapsto F\otimes p^*G.
\]

\begin{proposition}[Twisting embedding]\label{prop:twisting-embedding}
If $F$ is relatively simple, then $T_F$ is exact and fully faithful.  More precisely,
\[
\Hom_{Y_f}(F\otimes p^*G_1,F\otimes p^*G_2)
\cong\Hom_X(G_1,G_2),
\]
and there is a natural injection
\[
\Ext_X^1(G_1,G_2)\hookrightarrow
\Ext_{Y_f}^1(F\otimes p^*G_1,F\otimes p^*G_2).
\]
If a split centre exists, then
\[
\Lambda_\phi(F\otimes p^*G)
\cong(p_*F\otimes G,\rho_F\otimes\id_G).
\]
\end{proposition}

\begin{proof}
Put $H=G_1^\vee\otimes G_2$.  The relevant Hom sheaf on $Y_f$ is $\End(F)\otimes p^*H$.  By projection formula and relative simplicity,
\[
p_*(\End(F)\otimes p^*H)\cong H.
\]
Taking global sections gives the Hom isomorphism.  The Leray spectral sequence gives the initial exact sequence
\[
0\to H^1(X,H)\to H^1\bigl(Y_f,\End(F)\otimes p^*H\bigr)
\to H^0\bigl(X,R^1p_*\End(F)\otimes H\bigr),
\]
which yields the Ext injection.  Exactness is clear for tensoring by locally free sheaves, and the last statement is \cref{thm:clifford-dictionary}.
\end{proof}

\begin{theorem}[Wildness criterion]\label{thm:wildness-criterion}
Assume that $Y_f/X$ carries a relatively simple relatively Ulrich bundle $F$.  If there are vector bundles $A,B$ on $X$ such that
\[
\End(A)=\End(B)=k,
\quad\Hom(A,B)=\Hom(B,A)=0,
\quad\dim\Ext^1(B,A)\ge3,
\]
then $\Ulr(Y_f/X)$ is strictly wild.  More strongly, for every finitely generated $k$-algebra $\Lambda$ there is an exact fully faithful functor
\[
\operatorname{mod}\Lambda\longrightarrow\Ulr(Y_f/X).
\]
If $p$ is smooth and a split centre exists, the same is true of $\Rep_X(C_{f,\phi})$.
\end{theorem}

\begin{proof}
Compose \cref{lem:all-algebras-free-two,lem:free-to-kronecker,prop:extension-embedding,prop:twisting-embedding}.  Every functor in the composition is exact and fully faithful.  The Clifford-side statement follows from the equivalence in \cref{thm:clifford-dictionary}.
\end{proof}

\begin{corollary}[Curves of genus at least two]\label{cor:curve-wildness}
Let $X$ be a smooth projective curve of genus $g\ge2$ and let $F\in\Ulr(Y_f/X)$ be relatively simple.  Then $\Ulr(Y_f/X)$ is strictly wild.  Moreover, for every $N\ge1$ there exists a smooth irreducible parameter variety $T_N$ of dimension
\[
N^2(g-1)+1
\]
and a family of pairwise non-isomorphic indecomposable relatively Ulrich bundles $E_t$ such that
\[
\dim\Ext^1_{Y_f}(E_t,E_t)\ge N^2(g-1)+1.
\]
\end{corollary}

\begin{proof}
For strict wildness, take $A=\OO_X$ and let $B$ be a stable bundle of rank $3$ and degree $0$; such bundles exist on a smooth projective curve of genus at least two, see \cite{Ses82}.  Stable bundles of the same slope are simple, and a nonzero map between two stable bundles of the same slope is an isomorphism.  Since $A$ and $B$ have different ranks,
\[
\Hom(A,B)=\Hom(B,A)=0.
\]
Riemann--Roch gives
\[
\dim\Ext^1(B,A)=h^1(B^\vee)=3(g-1)\ge3.
\]
Apply \cref{thm:wildness-criterion}.

For the family statement, choose an integer $e$ coprime to $N$. The
moduli space of stable bundles of rank $N$ and degree $e$ on $X$ is
smooth and irreducible of dimension
\[
N^2(g-1)+1.
\]
In the coprime case one may choose a universal bundle after the usual
normalization; it is unique up to tensoring by a line bundle pulled
back from the parameter space. Choose such a family $G_t$ and set
\[
E_t=F\otimes p^*G_t.
\]
Each $E_t$ has rank
\[
N\,\rk F.
\]
By \cref{prop:twisting-embedding}, non-isomorphic stable bundles give non-isomorphic indecomposable relatively Ulrich bundles.  Finally
\[
\dim\Ext^1_{Y_f}(E_t,E_t)
\ge h^1(X,\End G_t).
\]
Since $G_t$ is stable, $h^0(\End G_t)=1$, while $\rk\End G_t=N^2$ and $\deg\End G_t=0$.  Riemann--Roch gives
\[
h^1(\End G_t)=N^2(g-1)+1.
\]
\end{proof}

\begin{remark}[What the wildness detects]\label{rem:fibre-invisible}
For a geometric point $\bar x$,
\[
(F\otimes p^*G)|_{Y_{\bar x}}\cong F_{\bar x}^{\oplus\rk G}.
\]
Thus all parameters introduced by $G$ disappear on a fixed fibre.  The wildness proved here comes from the base.  It does not imply that the individual fibre $Y_{\bar x}$ has wild Ulrich representation type.  This distinction is essential.
\end{remark}

\section{Quadric fibrations, monodromy, and Brauer obstructions}
\label{sec:quadric-complexity}

\paragraph{Purpose of the section.}
For quadric fibrations, the Clifford description is intrinsic and does
not require a projection centre.  We use the even Clifford algebra of
the line-bundle-valued quadratic form to compute the minimal possible
rank of a relatively Ulrich bundle.

The answer is controlled by two different phenomena.  In one parity,
the obstruction is a Brauer class on the base.  In the other parity,
the centre of the Clifford algebra is a discriminant double cover, and
monodromy may force the two fibrewise spinor types to occur together.

Throughout this section, assume that
\[
\operatorname{char}k\neq2,
\]
that $X$ is connected, and that
\[
p:Y=V(f)\longrightarrow X
\]
is a smooth quadric fibration of relative dimension
\[
n-1\geq1,
\]
where
\[
f\in
H^0\!\left(
X,\Sym^2E^\vee\otimes L
\right),
\qquad
\rk E=n+1.
\]

Put
\[
\mathscr C_0
:=
\Cl_0(E,f,L),
\]
the even Clifford algebra of the line-bundle-valued quadratic form
$(E,f,L)$.

By \cref{thm:quadric-clifford}, there is an equivalence
\[
\Ulr(Y/X)
\simeq
\Rep_X(\mathscr C_0),
\]
and if a relatively Ulrich bundle $F$ corresponds to the
$\mathscr C_0$-module $B$, then
\[
B\simeq p_*F,
\qquad
\rk B=2\rk F.
\]

When a split projection centre exists,
\cref{prop:quadric-centre-comparison} identifies
\[
C_{f,\phi}
\simeq
\mathscr C_0.
\]
Thus the intrinsic description used in this section extends the
projection-centre description of \cref{thm:clifford-dictionary}.

\begin{proposition}[Structure of the intrinsic even Clifford algebra]
\label{prop:intrinsic-clifford-structure}
The algebra
\[
\mathscr C_0=\Cl_0(E,f,L)
\]
is locally free of rank
\[
2^n
\]
as an $\OO_X$-module.

More precisely:

\begin{enumerate}
\item If $n$ is even, then $\mathscr C_0$ is an Azumaya
$\OO_X$-algebra of degree
\[
2^{n/2}.
\]
We denote its Brauer class by
\[
\alpha=[\mathscr C_0]\in\Br(X)[2].
\]

\item If $n$ is odd, the centre of $\mathscr C_0$ is a quadratic
\'etale $\OO_X$-algebra.  Let
\[
\tau:\widetilde X\longrightarrow X
\]
be the corresponding discriminant double cover.  Then $\mathscr C_0$,
viewed as an algebra over its centre, is Azumaya over $\widetilde X$
of degree
\[
2^{(n-1)/2}.
\]

If $\widetilde X$ is connected, we denote its Brauer class by
\[
\beta\in\Br(\widetilde X)[2].
\]

If
\[
\widetilde X\simeq X\sqcup X,
\]
we write
\[
\mathscr C_0\simeq A_+\times A_-
\]
and denote the two Brauer classes by
\[
\beta_+,\beta_-\in\Br(X)[2].
\]
Moreover, if the discriminant cover splits, then
  \[
   \beta_+=\beta_-.
   \]
\end{enumerate}
\end{proposition}

\begin{proof}
The PBW filtration for the line-bundle-valued Clifford algebra gives
\[
\rk_{\OO_X}\mathscr C_0=2^n.
\]
The description of the centre and the Azumaya degrees is the standard
structure theorem for a regular line-bundle-valued quadratic form;
see \cite{BK94,Knu91}.

If $n$ is even, then $E$ has odd rank. The canonical Clifford
involution identifies $\mathscr C_0$ with its opposite algebra.
Hence
\[
2[\mathscr C_0]=0
\]
in $\Br(X)$.

If $n$ is odd, then $E$ has even rank and the centre defines the
discriminant double cover
\[
\tau:\widetilde X\to X.
\]
Over this centre the even Clifford algebra is Azumaya, and its
Clifford Brauer class is $2$-torsion; see \cite{Auel15}.

If
\[
\widetilde X\simeq X\sqcup X,
\]
then
\[
\mathscr C_0\simeq A_+\times A_-.
\]
For a regular even-rank line-bundle-valued quadratic form with split
discriminant, Auel proves
\[
[A_+]=[A_-]\in\Br(X)[2];
\]
see \cite[Prop.~2.5]{Auel15}. Therefore
\[
\beta_+=\beta_-.
\]
\end{proof}

\subsection{Twisted sheaves and ranks of Clifford modules}

For a Brauer class
\[
\gamma\in\Br(T),
\]
let
\[
r(\gamma)
\]
denote the smallest positive rank of a locally free
$\gamma$-twisted sheaf on $T$.

\begin{proposition}[Ranks of modules over an Azumaya algebra]
\label{prop:azumaya-module-ranks}
Let $A$ be an Azumaya algebra of degree $m$ on a connected scheme
$T$, and let
\[
\gamma=[A]\in\Br(T).
\]
Then:

\begin{enumerate}
\item Left $A$-modules which are locally free over $\OO_T$ are
Morita-equivalent to locally free $(-\gamma)$-twisted sheaves.
If an $A$-module $P$ corresponds to a twisted sheaf of rank $s$, then
\[
\rk_{\OO_T}P=ms.
\]
Consequently, the smallest possible nonzero rank of a locally free
$A$-module is
\[
m\,r(\gamma).
\]

\item
\[
r(\gamma)=r(-\gamma),
\]
and
\[
r(\gamma)=1
\quad\Longleftrightarrow\quad
\gamma=0.
\]
Moreover,
\[
r(\gamma)\leq m.
\]

\item If $T$ is integral with generic point $\eta$, then
\[
\ind(\gamma_\eta)\mid r(\gamma).
\]

\item If $T$ is a smooth curve over the algebraically closed field
$k$, then
\[
\Br(T)=0.
\]

\item If $T$ is a smooth projective surface and
$\per(\gamma)$ is prime to $\operatorname{char}k$, then
\[
r(\gamma)=\per(\gamma).
\]
\end{enumerate}
\end{proposition}

\begin{proof}
For (1), represent the Brauer class by the gerbe on which there is a
locally free twisted sheaf $V$ of rank $m$ satisfying
\[
A\simeq\End(V).
\]
Morita equivalence is given by
\[
P\longmapsto\Hom_A(V,P),
\qquad
N\longmapsto V\otimes N.
\]
Locally, where
\[
A\simeq M_m(\OO_T),
\]
this gives
\[
\rk_{\OO_T}P=m\,\rk N.
\]

Duality gives
\[
r(\gamma)=r(-\gamma).
\]
A rank-one twisted sheaf trivializes the Brauer class, which proves
\[
r(\gamma)=1
\iff
\gamma=0.
\]
The twisted bundle defining the Azumaya algebra gives
\[
r(\gamma)\leq m.
\]

For (3), let $\eta$ be the generic point.  If the corresponding
central simple algebra has index $i$, then every twisted module over
the function field has rank divisible by $i$.  Hence
\[
\ind(\gamma_\eta)\mid r(\gamma).
\]

For a smooth curve over an algebraically closed field, the Brauer
group is zero by Tsen's theorem.

Finally, for a smooth projective surface, de Jong's period--index
theorem gives
\[
\ind(\gamma_\eta)=\per(\gamma)
\]
when the period is prime to the characteristic.  A twisted module of
minimal generic rank extends to a coherent twisted sheaf; taking the
reflexive hull gives a locally free twisted sheaf on the regular
surface.  Hence
\[
r(\gamma)=\per(\gamma).
\]
\end{proof}

\begin{proposition}[Odd Clifford rank in the presence of a projection centre]
\label{prop:odd-clifford-rank}
Assume that $n$ is odd and that $Y/X$ admits a split relative
projection centre $\phi$.

By \cref{prop:quadric-centre-comparison},
\[
\mathscr C_0
\simeq
C_{f,\phi}
\simeq
\Cl(W,Q),
\]
where $W$ has odd rank $n$.

Let
\[
\gamma
=
[\Cl_0(W,Q)]
\in\Br(X)[2].
\]
Then multiplication induces an isomorphism
\[
\Cl_0(W,Q)\otimes_{\OO_X}Z(Q)
\xrightarrow{\sim}
\Cl(W,Q).
\]

Consequently, if
\[
\tau:\widetilde X\longrightarrow X
\]
is the discriminant double cover and
\[
\beta\in\Br(\widetilde X)
\]
is the Brauer class of $\mathscr C_0$ over its centre, then
\[
\beta=\tau^*\gamma.
\]

In particular, if
\[
\widetilde X\simeq X\sqcup X,
\]
then
\[
\beta_+=\beta_-=\gamma.
\]

Moreover,
\[
r(\beta)\leq r(\gamma)\leq2r(\beta)
\]
when $\widetilde X$ is connected.
\end{proposition}

\begin{proof}
The assertion is local on $X$.  Choose an orthogonal basis
\[
e_1,\ldots,e_n
\]
of $W$ and put
\[
z=e_1\cdots e_n.
\]
Since $n$ is odd, $z$ is an odd central element of the Clifford
algebra.  Nondegeneracy of $Q$ implies that
\[
z^2
\]
is a unit.  Hence
\[
Z(Q)=\OO_X\oplus\OO_Xz.
\]

Multiplication by $z$ identifies the odd part with the even part:
\[
\Cl_1(W,Q)=\Cl_0(W,Q)\,z.
\]
Therefore
\[
\Cl(W,Q)
=
\Cl_0(W,Q)
\oplus
\Cl_0(W,Q)z,
\]
and multiplication gives
\[
\Cl_0(W,Q)\otimes_{\OO_X}Z(Q)
\xrightarrow{\sim}
\Cl(W,Q).
\]

It follows that the Azumaya algebra over the discriminant cover is
the pullback of $\Cl_0(W,Q)$.  Hence
\[
\beta=\tau^*\gamma.
\]

If
\[
\widetilde X=X\sqcup X,
\]
the restrictions to the two components are both
$\Cl_0(W,Q)$, so
\[
\beta_+=\beta_-=\gamma.
\]

Finally, pulling back a $\gamma$-twisted bundle gives a
$\beta$-twisted bundle of the same rank, and pushing forward a
$\beta$-twisted bundle along the degree-two map
$\widetilde X\to X$ doubles its rank.  Therefore
\[
r(\beta)\leq r(\gamma)\leq2r(\beta).
\]
\end{proof}

\begin{theorem}[Relative Ulrich complexity of a quadric fibration]
\label{thm:quadric-complexity}
Let
\[
p:Y\longrightarrow X
\]
be as above.  Then the following statements hold.

\begin{enumerate}
\item For every geometric point $\bar x\to X$,
\[
\uc(Y_{\bar x})
=
2^{\lfloor (n-2)/2\rfloor}.
\]

\item If $n$ is even, let
\[
\alpha=[\mathscr C_0]\in\Br(X)[2].
\]
Then
\[
\uc(Y/X)
=
2^{n/2-1}r(\alpha)
=
\uc(Y_{\bar x})\,r(\alpha).
\]

\item If $n$ is odd and the discriminant double cover
\[
\widetilde X\longrightarrow X
\]
is connected, let
\[
\beta\in\Br(\widetilde X)[2]
\]
be the Brauer class of $\mathscr C_0$ over its centre.  Then
\[
\uc(Y/X)
=
2^{(n-1)/2}r(\beta)
=
2\,r(\beta)\,\uc(Y_{\bar x}).
\]

\item If $n$ is odd and
\[
\widetilde X\simeq X\sqcup X,
\]
with component Brauer classes
\[
\beta_+,\beta_-\in\Br(X)[2],
\]
then
\[
\uc(Y/X)
=
\uc(Y_{\bar x})
\min\{r(\beta_+),r(\beta_-)\}.
\]
\end{enumerate}

In every case,
\[
\uc(Y/X)\leq2^{n-1}.
\]
\end{theorem}
\begin{proof}
By \cref{thm:quadric-clifford}, there is an equivalence
\[
\Ulr(Y/X)
\simeq
\Rep_X(\mathscr C_0),
\]
and a relatively Ulrich bundle of rank $r$ corresponds to a
$\mathscr C_0$-module whose underlying $\OO_X$-bundle has rank
\[
2r.
\]

We first compute the fibrewise complexity.

Let $\bar x\to X$ be a geometric point.  If $n$ is even, then
\[
(\mathscr C_0)_{\bar x}
\simeq
M_{2^{n/2}}(k),
\]
so its smallest nonzero module has dimension
\[
2^{n/2}.
\]
Since the coefficient rank is twice the Ulrich rank, the smallest
Ulrich rank is
\[
2^{n/2-1}.
\]

If $n$ is odd, then
\[
(\mathscr C_0)_{\bar x}
\simeq
M_{2^{(n-1)/2}}(k)
\times
M_{2^{(n-1)/2}}(k).
\]
Thus the smallest nonzero module has dimension
\[
2^{(n-1)/2},
\]
and the corresponding Ulrich rank is
\[
2^{(n-3)/2}.
\]

In both cases,
\[
\uc(Y_{\bar x})
=
2^{\lfloor(n-2)/2\rfloor}.
\]

Now suppose that $n$ is even.  By
\cref{prop:intrinsic-clifford-structure},
$\mathscr C_0$ is Azumaya of degree
\[
m=2^{n/2}
\]
with Brauer class $\alpha$.  By
\cref{prop:azumaya-module-ranks}, the smallest nonzero
$\mathscr C_0$-module has $\OO_X$-rank
\[
m\,r(\alpha).
\]
Dividing by $2$, because coefficient rank is twice Ulrich rank, gives
\[
\uc(Y/X)
=
2^{n/2-1}r(\alpha)
=
\uc(Y_{\bar x})r(\alpha).
\]

Next suppose that $n$ is odd and $\widetilde X$ is connected.
The algebra $\mathscr C_0$ is Azumaya over $\widetilde X$ of degree
\[
m=2^{(n-1)/2}.
\]
If $F$ is relatively Ulrich of rank $r$, its coefficient bundle
$B=p_*F$ has rank
\[
2r
\]
over $X$.  Since $B$ is naturally a module over the centre
$\OO_{\widetilde X}$ and
\[
[\widetilde X:X]=2,
\]
it has rank
\[
r
\]
as an $\OO_{\widetilde X}$-bundle.

Therefore the smallest possible value of $r$ is
\[
m\,r(\beta)
=
2^{(n-1)/2}r(\beta).
\]
Since
\[
2^{(n-1)/2}
=
2\,\uc(Y_{\bar x}),
\]
we obtain
\[
\uc(Y/X)
=
2r(\beta)\uc(Y_{\bar x}).
\]

Finally suppose that
\[
\widetilde X=X\sqcup X.
\]
Then
\[
\mathscr C_0=A_+\times A_-.
\]
A nonzero module may be supported on either component.  The least
possible coefficient rank is therefore
\[
2^{(n-1)/2}
\min\{r(\beta_+),r(\beta_-)\}.
\]
Dividing by $2$ gives
\[
\uc(Y/X)
=
2^{(n-3)/2}
\min\{r(\beta_+),r(\beta_-)\}.
\]
Since
\[
2^{(n-3)/2}=\uc(Y_{\bar x}),
\]
this proves (4).

The regular module of $\mathscr C_0$ has rank
\[
2^n.
\]
Hence, by \cref{thm:quadric-clifford}, it corresponds to a relatively
Ulrich bundle of rank
\[
2^{n-1}.
\]
This proves the final bound.
\end{proof}
\begin{corollary}[Rank divisibility]
\label{cor:quadric-rank-divisibility}
Assume in addition that $X$ is integral.

\begin{enumerate}
\item If $n$ is even and $\eta$ is the generic point of $X$, then
every relatively Ulrich bundle $F$ satisfies
\[
\uc(Y_{\bar\eta})\,
\ind(\alpha_\eta)
\mid
\rk F.
\]

\item If $n$ is odd, the discriminant cover $\widetilde X$ is
connected and integral, and $\widetilde\eta$ is its generic point,
then every relatively Ulrich bundle $F$ satisfies
\[
2\,\uc(Y_{\bar\eta})\,
\ind(\beta_{\widetilde\eta})
\mid
\rk F.
\]
\end{enumerate}
\end{corollary}

\begin{proof}
Let
\[
r=\rk F
\]
and let
\[
B=p_*F.
\]
Then
\[
\rk_XB=2r.
\]

Suppose first that $n$ is even.  Over the function field $k(X)$, the
algebra
\[
(\mathscr C_0)_\eta
\]
is a central simple algebra of degree
\[
m=2^{n/2}.
\]
Let
\[
i=\ind(\alpha_\eta).
\]
Every finite module over this central simple algebra has dimension
divisible by
\[
mi.
\]
Hence
\[
mi\mid2r.
\]
Dividing by $2$ gives
\[
2^{n/2-1}i\mid r.
\]
Since
\[
2^{n/2-1}=\uc(Y_{\bar\eta}),
\]
the first assertion follows.

Now suppose that $n$ is odd and the discriminant cover is connected.
Over $k(\widetilde X)$, the coefficient module has dimension $r$.
The Clifford algebra over this field is central simple of degree
\[
m=2^{(n-1)/2}.
\]
Therefore
\[
m\ind(\beta_{\widetilde\eta})\mid r.
\]
Since
\[
m=2\,\uc(Y_{\bar\eta}),
\]
the second assertion follows.
\end{proof}
\begin{corollary}
\label{cor:quadric-consequences}
In the setting of \cref{thm:quadric-complexity}:

\begin{enumerate}
\item If $X$ is a smooth curve, then
\[
\uc(Y/X)=\uc(Y_{\bar x})
\]
except when $n$ is odd and the discriminant double cover is connected.
In that case
\[
\uc(Y/X)=2\uc(Y_{\bar x}).
\]

\item If $X$ is a smooth projective surface and $n$ is even, then
\[
\uc(Y/X)
=
\per(\alpha)\,\uc(Y_{\bar x}),
\]
where
\[
\per(\alpha)\in\{1,2\}.
\]

\item One has
\[
\uc(Y/X)=\uc(Y_{\bar x})
\]
if and only if one of the following holds:
\[
\begin{cases}
n\text{ is even and }\alpha=0;\\[2mm]
n\text{ is odd, }\widetilde X\simeq X\sqcup X,
\text{ and at least one of }\beta_+,\beta_-\text{ is zero}.
\end{cases}
\]

\item If, in addition, $n$ is odd and a split projection centre
exists, then in the split-discriminant case
\[
\beta_+=\beta_-=\gamma
\]
by \cref{prop:odd-clifford-rank}.  Hence in this case
\[
\uc(Y/X)=\uc(Y_{\bar x})
\iff
\gamma=0.
\]
\end{enumerate}
\end{corollary}

\begin{proof}
For a smooth curve over an algebraically closed field, the Brauer
group is zero.  The same is true for every connected finite
\'etale cover of such a curve.  Thus
\cref{thm:quadric-complexity} gives (1).

For a smooth projective surface, the Clifford Brauer class is
$2$-torsion.  Since $2$ is invertible, de Jong's period--index theorem
and \cref{prop:azumaya-module-ranks} give
\[
r(\alpha)=\per(\alpha),
\]
which proves (2).

Statement (3) follows directly from the three cases of
\cref{thm:quadric-complexity}.  If the discriminant cover is connected
and $n$ is odd, the unavoidable factor $2$ shows that equality with
the fibrewise complexity is impossible.

Finally, (4) follows from
\cref{prop:odd-clifford-rank}.
\end{proof}
\begin{remark}[Brauer obstruction versus monodromy]
\label{rem:brauer-monodromy}
The formulas in \cref{thm:quadric-complexity} separate two different
obstructions.

When $n$ is even, the even Clifford algebra is Azumaya directly over
$X$.  There is a single simple Clifford-module type on each geometric
fibre.  The obstruction to realizing the smallest fibrewise rank
globally is therefore the Brauer class
\[
\alpha\in\Br(X)[2].
\]

When $n$ is odd, every geometric fibre has two simple Clifford-module
types, corresponding geometrically to the two half-spinor families.
These two types are parametrized by the discriminant double cover
\[
\widetilde X\longrightarrow X.
\]
If this cover is connected, monodromy exchanges the two fibrewise
types.  A global object must therefore contain both, which explains
the extra factor
\[
2
\]
in
\[
\uc(Y/X)
=
2r(\beta)\uc(Y_{\bar x}).
\]

Thus a minimal Ulrich bundle on every geometric fibre need not descend
to a relatively Ulrich bundle of the same rank.  The failure is
measured precisely by Brauer classes and by monodromy of the
discriminant cover.
\end{remark}

\section{Examples}

\paragraph{Purpose of the section.}
We now test the general results on explicit families.  The examples isolate four phenomena: normalized Ulrich line bundles on cyclic plane curves, nontrivial coefficient bundles on a relative conic, monodromy without a fibrewise obstruction, and a genuine Brauer obstruction.

\subsection{Cyclic plane curves}
\begin{example}[A normalized Ulrich line bundle in every degree]\label{ex:cyclic-plane}
Assume $d$ is invertible in $k$.  Let $S\in\Pic(X)$, let $V$ have rank $2$, put $W=V\otimes S^{-1}$, and choose
\[
\ell_1,\dots,\ell_d\in H^0(X,W^\vee)
\]
such that $\ell_i\wedge\ell_j$ is nowhere zero for $i\ne j$.  Define
\[
Y=\{w^d=\ell_1\cdots\ell_d\}\subset\PP_X(S\oplus V).
\]
Then $Y\to X$ is smooth, and $Y$ carries a normalized relatively Ulrich line bundle.  Consequently
\[
\uc(Y/X)=\ucnor(Y/X)=1,
\qquad \cc=d.
\]
\end{example}

\begin{proof}
On a geometric fibre, a singular point must have $w=0$ and one $\ell_i=0$.  Pairwise independence implies that all other $\ell_j$ are nonzero there, so the derivative of $\ell_1\cdots\ell_d$ is nonzero.  Thus the fibre is smooth.

Let $B=\OO_X^{\oplus d}$ with basis $b_1,\dots,b_d$.  Define $\mu:W\to\End(B)$ by the cyclic rule
\[
\mu(v)b_{j+1}=\ell_j(v)b_j\quad(1\le j<d),
\qquad
\mu(v)b_1=\ell_d(v)b_d.
\]
Going once around the cycle gives
\[
\mu(v)^d=\prod_{i=1}^d\ell_i(v)\,I.
\]
Thus $B$ is a rank-$d$ representation of the generalized Clifford algebra.  By \cref{cor:relative-cyclic}, it corresponds to a relative Ulrich bundle of rank one; since $B$ is trivial, the bundle is normalized.
\end{proof}

\begin{remark}
The rank-two hypothesis on $W$ is essential for smoothness of this product construction.  If $\rk W\ge3$, two branch hyperplanes meet, and their intersection produces singular points with $w=0$.
\end{remark}

\subsection{The relative Veronese conic}
\begin{example}\label{ex:veronese-setup}
Assume $\charac k\ne2$.  Let $M\in\Pic(X)$ and
\[
N=\OO_X\oplus M^{-1}.
\]
The relative Veronese embedding
\[
\PP_X(N)\hookrightarrow\PP_X(\Sym^2N)
\]
identifies $\PP_X(N)$ with the discriminant conic
\[
b^2-4ac=0.
\]
The induced polarization is
\[
\OO_{\PP(N)}(2).
\]
\end{example}

\begin{proposition}[Classification on the relative Veronese conic]\label{prop:veronese-classification}
Let $p:\PP_X(N)\to X$ be as above.
\begin{enumerate}
\item The functor
\[
\VB(X)\longrightarrow\Ulr(\PP_X(N)/X),
\qquad
G\longmapsto\OO_{\PP(N)}(1)\otimes p^*G
\]
is an equivalence of exact categories.  In particular $\uc=1$.
\item For every $G$,
\[
p_*\bigl(\OO_{\PP(N)}(1)\otimes p^*G\bigr)
\cong N^\vee\otimes G
\cong G\oplus(M\otimes G).
\]
\item If $X$ is proper and connected and $M\not\cong\OO_X$, then no nonzero relatively Ulrich bundle is normalized; hence
\[
\ucnor=\infty.
\]

\item The projection-centre Clifford algebra is split in the Brauer sense. More precisely,
\[
C_{f,\phi}
\simeq
\operatorname{End}_{\mathcal O_X}
\bigl(\mathcal O_X\oplus M\bigr).
\]
Thus it is locally a \(2\times2\) matrix algebra, but it need not
be globally isomorphic to
\(\operatorname{Mat}_2(\mathcal O_X)\).
\end{enumerate}
\end{proposition}

\begin{proof}
A fibre is $\PP^1$ polarized by $\OO(2)$.  By Birkhoff--Grothendieck, a vector bundle on the fibre is a sum of line bundles $\OO(a_i)$; it is Ulrich for $\OO(2)$ exactly when
\[
H^*(\PP^1,\OO(a_i-2))=0,
\]
which forces $a_i=1$.  Therefore, if $F$ is relatively Ulrich, then
\[
F':=F\otimes\OO_{\PP(N)}(-1)
\]
is trivial on every fibre.  The bundle $G=p_*F'$ is locally free, and the adjunction $p^*G\to F'$ is an isomorphism fibrewise and hence globally.  This proves essential surjectivity; full faithfulness follows from $p_*\OO_{\PP(N)}=\OO_X$.

For (2), our projective-bundle convention gives
\[
p_*\OO_{\PP(N)}(1)=N^\vee=\OO_X\oplus M.
\]
For (3), if
\[
G\oplus(M\otimes G)\cong\OO_X^{\oplus2r},
\]
Krull--Schmidt for vector bundles on a proper scheme \cite{Ati56} implies that all indecomposable summands on the left are copies of $\OO_X$.  Thus $G\cong\OO_X^{\oplus r}$ and $M^{\oplus r}\cong\OO_X^{\oplus r}$, forcing $M\cong\OO_X$, a contradiction.

For (4), Let

\[
F_0=\mathcal O_{\mathbb P(N)}(1).
\]
Its coefficient bundle is

\[
p_*F_0\simeq\mathcal O_X\oplus M.
\]
The projection-centre correspondence therefore gives an algebra
homomorphism

\[
C_{f,\phi}
\longrightarrow
\operatorname{End}_{\mathcal O_X}
\bigl(\mathcal O_X\oplus M\bigr).
\]
On every geometric fibre, the source is the split rank-two
Clifford algebra, isomorphic to \(\operatorname{Mat}_2\), and the
target is the endomorphism algebra of its two-dimensional simple
module. The displayed homomorphism is consequently an isomorphism
on every geometric fibre. Both sides are locally free of rank
\(4\), so it is an isomorphism globally.

It follows that \(C_{f,\phi}\) is Brauer-trivial. This does not
assert that its splitting bundle \(\mathcal O_X\oplus M\) is
trivial.

\end{proof}

\begin{remark}[Concrete matrix]
In the ordered local basis \((s,t)\), we use the column-vector
convention: the columns of the matrix are the coordinates of the
images of the basis elements. Since
\[
a=s^2,\qquad c=t^2,\qquad w=2st,
\]
we have
\[
ws=2at,\qquad wt=2cs.
\]
Therefore multiplication by \(w\) is represented by
\[
M_w=
\begin{pmatrix}
0 & 2c\\
2a & 0
\end{pmatrix}.
\]
In particular,
\[
M_w^2=4ac\,I_2,
\]
as required by the defining equation \(w^2-4ac=0\).
\end{remark}

\subsection{Monodromy on a quadric surface bundle}
\begin{proposition}\label{prop:monodromy-example}
Assume $\charac k\ne2$.  Let $X$ be integral and let $\eta\in\Pic(X)[2]$ be nontrivial.  Choose $\tau:\eta^{\otimes2}\cong\OO_X$, set
\[
W=\OO_X^{\oplus2}\oplus\eta,
\qquad
Q(x,y,z)=x^2+y^2+\tau(z^2),
\]
and let
\[
Y=\{w^2=Q\}\subset\PP_X(\OO_X\oplus W).
\]
Then:
\begin{enumerate}
\item $Y\to X$ is a smooth quadric-surface bundle with a split centre and every geometric fibre has Ulrich complexity $1$;
\item the discriminant double cover is connected, so $\uc(Y/X)\ge2$;
\item if $X$ is a smooth curve, then $\uc(Y/X)=2$.  There is a rank-two relatively Ulrich bundle $F$ that is irreducible in the relative exact category, although
\[
F_{\bar x}\cong\OO(1,0)\oplus\OO(0,1)
\]
for every geometric fibre.
\end{enumerate}
\end{proposition}

\begin{proof}
The form $Q$ is nondegenerate, proving smoothness.  Here $n=3$, so a geometric fibre is a quadric surface and has complexity $1$ by \cref{thm:quadric-complexity}.

Let $e_1,e_2$ be the standard basis of $\OO_X^{\oplus2}$.  In the Clifford algebra the element $e_1e_2z$, with $z$ a local section of $\eta$, is central; its square is $-\tau(z^2)$.  Thus the centre is the quadratic etale algebra
\[
\OO_X\oplus\eta
\]
with the corresponding multiplication.  If its spectrum were disconnected over connected $X$, it would be $X\sqcup X$, hence would admit a section; such a section would trivialize $\eta$.  Therefore the discriminant cover $\widetilde X\to X$ is connected.  The odd-rank formula in \cref{thm:quadric-complexity} gives $\uc(Y/X)=2r(\beta)\ge2$.

If $X$ is a smooth curve, so is the etale cover $\widetilde X$, and $\Br(\widetilde X)=0$.  Hence the Clifford algebra is $\End_{\widetilde X}(V)$ for a rank-two bundle $V$ on $\widetilde X$.  Viewed as a Clifford module on $X$, its underlying $\OO_X$-rank is $4$, so the corresponding relative Ulrich bundle has rank $2$.  On a geometric fibre the two points of the discriminant cover give the two simple Clifford modules, hence the two Ulrich line bundles $\OO(1,0)$ and $\OO(0,1)$.  Globally, however, an admissible Clifford submodule would give a locally free submodule of the simple Morita module $V$ over the connected centre; its rank is therefore $0$ or full.  Thus the relative object is irreducible.
\end{proof}

\begin{example}
If $X$ is an elliptic curve with a nontrivial $2$-torsion line bundle $\eta$, \cref{prop:monodromy-example} gives a smooth quadric-surface bundle whose fibres have Ulrich complexity $1$ but whose relative Ulrich complexity is $2$.  This is the smallest-dimensional projective-base example of the monodromy jump in the theorem.
\end{example}

\subsection{A Brauer obstruction on an abelian surface}
\begin{proposition}\label{prop:brauer-example}
Assume $\charac k\ne2$.  Let $A$ be an abelian surface with
\[
\mathrm{NS}(A)=\mathbb Z\theta
\]
for a principal polarization $\theta$.  Let $\eta_1,\eta_2\in\Pic^0(A)[2]$ be linearly independent and choose trivializations $\eta_i^{\otimes2}\cong\OO_A$.  Put
\[
W=\eta_1\oplus\eta_2,
\qquad Q=\tau_1\perp\tau_2,
\]
and define the conic bundle
\[
Y=\{w^2=Q\}\subset\PP_A(\OO_A\oplus W).
\]
Then every geometric fibre is $\PP^1$ and has Ulrich complexity $1$, but
\[
\uc(Y/A)=2.
\]
Thus every fibre has an Ulrich line bundle while no relative Ulrich line bundle exists.
\end{proposition}

\begin{proof}
Here $n=2$.  By \cref{thm:quadric-complexity},
\[
\uc(Y/A)=r(\alpha),
\qquad \alpha=[\Cl(W,Q)]\in\Br(A)[2].
\]
The Clifford algebra is quaternionic, so $r(\alpha)\in\{1,2\}$.  It remains to prove $\alpha\ne0$.

Let $K=k(A)$.  Since $A$ is regular integral, the natural map $\Br(A)\to\Br(K)$ is injective \cite{Gro68}.  Choose rational generators $s_i$ of $\eta_i$ and write
\[
a_i=\tau_i(s_i^2)\in K^*.
\]
The generic Clifford algebra is the quaternion algebra $(a_1,a_2)_K$, whose $2$-torsion Brauer class is the Kummer cup product $(a_1)\cup(a_2)$; see \cite{GS06}.

Let $\kappa_i\in H^1(A,\mu_2)$ be the class of the etale double cover defined by $(\eta_i,\tau_i)$.  The Kummer exact sequence gives
\[
0\longrightarrow\Pic(A)/2\longrightarrow H^2(A,\mu_2)
\longrightarrow\Br(A)[2]\longrightarrow0.
\]
The generic restriction of $\kappa_1\cup\kappa_2$ is the class of $(a_1,a_2)$, so injectivity of $\Br(A)\to\Br(K)$ shows that $\alpha$ is the image of $\kappa_1\cup\kappa_2$.

Because multiplication by $2$ is surjective on $\Pic^0(A)$ and $\mathrm{NS}(A)=\mathbb Z\theta$,
\[
\Pic(A)/2=\mathrm{NS}(A)/2=\{0,c_1(\theta)\}.
\]
For $2$ invertible, the etale cohomology algebra of an abelian variety is the exterior algebra on $H^1$; see \cite{Mil86}.  Thus $\kappa_1\cup\kappa_2$ is a nonzero decomposable alternating $2$-form and has rank $2$.  The class $c_1(\theta)$ corresponds to the Weil pairing of the principal polarization and is nondegenerate on the four-dimensional space $A[2]$, hence has rank $4$; see \cite{Mum70}.  Therefore
\[
\kappa_1\cup\kappa_2\notin\{0,c_1(\theta)\}.
\]
Its Brauer image is nonzero, so $\alpha\ne0$.  Hence $r(\alpha)=2$ and $\uc(Y/A)=2$.
\end{proof}

\section{Acknowledgements}
The authors express their sincere gratitude to Professor Emre Co\c{s}kun and Professor A J Parameswaran for their valuable suggestions, which significantly improved the exposition of this manuscript. The first author is supported by a Postdoctoral Fellowship at IISER Berhampur, India, and gratefully acknowledges the institute for its excellent research environment. The second author acknowledges the financial support of the National Board for Higher Mathematics (NBHM), Government of India, under Fellowship No. 0203/13(35)/2021-RD-II/13162, and thanks IISER Berhampur for providing the facilities to carry out this work.

\end{document}